\documentclass[11pt]{article}
\usepackage{float}
\usepackage{amsmath}
\usepackage{amssymb}
\usepackage{algorithmic}
\usepackage{algorithm}
\usepackage{mathtools}
\usepackage{color}
\usepackage{tikz}
\usetikzlibrary{arrows,
	arrows.meta,
	bending,calc,patterns,angles,quotes}
\usepackage{float}

\newtheorem{theorem}{Theorem}

\newtheorem{lemma}{Lemma}
\newtheorem{proposition}{Proposition}
\newtheorem{remark}{Remark}

\def\endpf{{\ \hfill\hbox{\vrule width1.0ex height1.0ex}\parfillskip 0pt
}}
\newenvironment{proof}{\noindent{\bf Proof:}}{\endpf}
\begin{document}
	\title{Regulation of a single-server queue with customers who dynamically choose their service durations.}
	\author{Royi Jacobovic\thanks{Department of Statistics and Data-Science; The Hebrew University of Jerusalem; Jerusalem 9190501; Israel.
			{\tt royi.jacobovic@mail.huji.ac.il}}}\maketitle
	\begin{abstract}
		In recent years, there is a growing research about queueing models with customers who choose their service durations. In general, the model assumptions in the existing literature imply that every customer knows his service demand when he enters into the service position. Clearly, this property is not consistent with some real-life situations. Thus, motivated by this issue,  the current work includes a single-server queueing model with customers who \textit{dynamically} choose their service durations. In this setup, it is shown how to derive a price-regulation (Pigouvian tax) which implies an optimal resource allocation from a social point of view. In particular, it is explained how to compute the parameters of this price function efficiently. 
	\end{abstract}
	
	
	\bigskip
	\noindent {\bf AMS Subject Classification (MSC2010):} 60K25, 90B22, 90C15.

	\section{Introduction}
	In classical queueing models like \cite{leeman1964,Naor1969,Edelson1975}, it is assumed that every customer has a linear loss in his waiting time and a positive value which is associated with a service completion. Then, customers should decide whether to join or not to join the queue. In particular, once they join the queue, they have no influence on their service times. In such models, it is usually shown that once there is no regulation and customers act independently in order to maximize their own expected utility, the long-run average social welfare is not optimized. This phenomenon is known in economic literature as \textit{the tragedy of the commons} \cite{Hardin1968}. On the other hand,  once customers are required to pay a proper price for joining the queue, then the long-run average social welfare is maximized. This (optimal) price function which is used in order to restore social optimality is called a \textit{Pigouvian tax} (see, \textit{e.g.,} Chapter 11 of \cite{MWG1995}).   
	
	Especially, there is a literature regarding models in which customers have non-identical preferences. Models in this direction are associated with \textit{heterogeneous customers} and they are surveyed in \textit{e.g.,} Subsections 2.5, 3.3, 3.4 of \cite{Hassin2003}. Recently, several authors \cite{Hopp,Tong2014,Feldman2019,Oz2019}\footnote{\cite{Oz2019} is a work which was presented in The 20th INFORMS Applied Probability Society Conference, July 3-5, 2019, Brisbane Australia.} considered queueing models with heterogeneous customers such that each customer picks his own service duration. \cite{Hopp} motivated this assumption by discretionary tasks which exist in \textit{e.g.,} a call centre agent of a mail-order catalogue company who processes orders and is also tasked with up-selling. \cite{Tong2014} suggested a motivation which is based on a firm that specializes in designing websites. This firm wishes to maximize its long run average rate of revenues under the constraint that the expected net utility of any joining customer is nonnegative. \cite{Feldman2019} provides plenty of motivations including drivers that secure a parking spot choose how long they want to park, gym goers choose their training length, but dislike waiting for the equipment to be available, customers who come to a coffee shop to work enjoy using the wifi and choose the length of their stay, but dislike waiting for a table to free up, \textit{etc}. In addition, \cite{Feldman2019} includes an analysis of the impact of different price schemes (per-use fees, price rates, time limits) on the consumer surplus. In particular, in each of the papers \cite{Hopp,Tong2014,Feldman2019}, the utility function of a generic customer from service duration belongs to a very specific parametric class. This motivated \cite{Oz2019} to consider a broader domain of price functions. Namely, this reference analysed a similar model with a quasi-linear utility function satisfying certain regularity conditions and it is shown that externalities are internalized by the optimal price function. In other words, when the optimal price function is implemented, then every customer pays the expected loss which is inflicted on the society by his service demand. 
	
	Importantly, in all of these papers, the individual optimizations of the customers are \textit{static}. This means that each customer knows his service demand before the initiation of his service period.
	From a modelling perspective, this property does not make sense when considering some real-life situations. For example, consider a businessman who parks his car in a CBD of a big city. When he arrives, he reports his arrival in a cellphone application. Then, he goes to do business and dynamically decides when to leave the CBD. Notably, his decision might be influenced by some dynamic random effects such as meetings which are surprisingly cancelled, long waiting time in a bank, \textit{etc}. Finally, when he is about to leave, he indicates the cell phone application and automatically pays the bill. Another relevant example is of a couple who sit in a coffee-shop for a blind-date. No doubt that when they enter the place, they do not know exactly how long they are going to stay. Similarly, it is reasonable to think about a queue to a public restroom in a museum. Albeit it is a bit anachronistic example, think also about a queue of people to a public payphone.
	
	Motivated by this issue, in this work the customers' marginal utilities are nonincreasing right-continuous stochastic processes. Then, each customer who receives service, observes the evolution of his marginal utility over time and \textit{dynamically} decides when to depart. 
	
	The nonincreasing marginal utility assumption is consistent with Gossen's first law of diminishing marginal returns \cite{Gossen1983}. This means that the customers' preferences are convex. For a discussion regarding the intuition which stands behind economic models with agents having convex preferences, see, \textit{e.g.,} Lecture 4 of \cite{Rubinstein2012}. In addition, the assumption that customers' preferences are described by a stochastic process  is a natural generalization of the concept of heterogeneous customers. Furthermore, an empirical justification for models with agents having stochastic preferences is provided in, \textit{e.g.,} \cite{Agranov2017,Ballinger1997,Hey1995,Tversky1969}.    
	
	The main result of the current work is as follows: There exists a quadratic price function which is optimal and internalizes the externalities in this model. Especially, relative to existing literature, in this work the space of candidate price functions is much more general. Namely, in the current work candidate price functions are not required to satisfy (almost) any regularity conditions  \textit{e.g.,} nonnegativity, monotonicity, smoothness, convexity, \textit{etc}. Importantly, it is shown how to execute an efficient computation of the parameters of the optimal price function. 
	
	Since the current model assumptions are relatively less restrictive, it is impossible to derive the results by the techniques which were applied in \cite{Hopp,Tong2014,Feldman2019,Oz2019}. This motivates an application of the following new approach: Primarily, it is shown that every candidate price function implies a stable $M/G/1$ queue with a service distribution which is determined endogenously by the price function through the individual optimizations of the customers. Thus, instead of maximizing the long-run average social welfare objective functional over the space of candidate price functions, it is suggested to use two stages: First, conduct an \textit{optimization} of the objective functional over a space of nonnegative probability measures (service distributions) for which the resulting $M/G/1$ is stable. Then, conduct a \textit{reverse-engineering} of the optimal price function, \textit{i.e.,} guess a price function which makes the customers choose service times having the same distribution like the solution of the optimization.   
	
	The rest of this work is organized as follows: Section \ref{sec: model description} is dedicated for a precise description of the model in which customers may neither balk nor renege. Section \ref{subsec: social planner} includes a precise statement of the social planner's problem in this model. Section \ref{sec: price function} includes the main result of this work with a discussion. Section \ref{sec: proof} is dedicated for explanation of the proof's guidelines while the technical details are given in an Appendix. Section \ref{sec: numerical procedure} is about an efficient computation of the parameters of the optimal price function. In particular, this section includes an analysis of some special cases. Section \ref{sec: balking} is focused on a similar model in which customers may balk. Section \ref{sec: retrial M/G/1} is about an analogue model with costumers' retrials. Especially, the results regarding this system motivate a conjecture about an expression of the expected externalities in an $M/G/1$ retrial queue with no waiting room, infinite orbit capacity and exponential retrial times. Finally, Section \ref{sec: future research} includes some open questions for future research. 	
	
	\newpage
	This paragraph includes some notations to be used later on. For every $(a,b)\in\mathbb{R}^2$ let $a\wedge b\equiv\min\{a,b\}$, $a^+\equiv\max\{a,0\}$ and $a^-\equiv (-a)^+$. In addition, $(\Omega,\mathcal{F},P )$ is the probability space which is in the background of the current probabilistic discussion. Moreover, $X\in\mathcal{F}$ means that $X$ is $\mathcal{F}$-measurable and for every such $X$, once exists, $EX$ is the expectation of $X$ with respect to (w.r.t.) $P$. For simplicity of notation, if there is an equality (inequality) of two random variables without further description, then it should be interpreted as a pointwise equality (inequality). Finally, for every $A\subseteq\mathbb{R}$ denote 
	\[1_A(x)\equiv
	\begin{dcases}
	1 & x\in A \\
	0 & x\notin A 
	\end{dcases}\ \ , \ \ \forall x\in\mathbb{R}\,.
	\]

	\section{Model description} \label{sec: model description}
	Consider a single-server service facility with an infinite waiting room. Customers arrive at this facility according  to a Poisson process with rate $\lambda\in(0,\infty)$. In addition, the service discipline is work-conserving and non-preemptive such that  customers are those who determine service durations. Let $X(\cdot),X_1(\cdot),X_2(\cdot),\ldots$ be an iid sequence of random processes which is independent from the arrival process. In addition, assume that $\left\{ X(s);s\geq0\right\}$ is a nonincreasing right-continuous stochastic process for which $X(0)$ is a positive square-integrable random variable. Moreover, let $C,C_1,C_2,\ldots$ be an iid sequence of positive random variables with mean $\gamma\in\left(0,\infty\right)$. It is assumed that this sequence is independent from all other random elements in this model. Then, for every $i\geq1$, the utility of the $i$'th customer from waiting $w\geq0$ minutes (excluding service time) and receiving a service of $s\geq0$ minutes equals to 
	\begin{equation}\label{eq: customer surplus}
	U_i\left(s,w;p\right)\equiv  \int_0^sX_i(t)dt-p(s)-C_iw\,.
	\end{equation}
	This means that the $i$'th customer suffers a linear loss from waiting time (with rate $C_i$) and has a stochastic marginal utility from service duration which is given by $ X_i(\cdot)$. Note that $X_i(\cdot)$ is nonincreasing, \textit{i.e.}, the $i$'th customer has a stochastic concave utility function of service duration. In addition, $p(s)$ is a deterministic payment which has to be paid by a customer who occupies the server for $s$ minutes. As to be explained in Section \ref{subsec: social planner},  $p(\cdot)$ is going to be determined by a social planner. 
	
	Let $\xi(\cdot)$ be a nondecreasing right-continuous nonnegative deterministic function such that $\xi(0)=0$. In particular, notice that it is possible to have $\xi(\cdot)$ which is identically zero. Then, assume that the server's revenue from providing service of $s\geq0$ minutes to an arbitrary customer (regardless his identity) is given by
	\begin{equation}\label{eq: server surplus}
	R(s;p)\equiv p(s)-\int_0^s\xi(t)dt\,.
	\end{equation}
	Moreover, assume that the server's revenue due to idle periods is identically zero. Therefore, if the $i$'th customer waits $w\geq0$ minutes and gets service of $s\geq0$ minutes, then a summation of \eqref{eq: customer surplus} and \eqref{eq: server surplus} implies that the social welfare which is gained due to the sojourn of this customer equals to
	\begin{equation} \label{eq: social welfare}
	SW_i(s,w)\equiv\int_0^sV_i(t)dt-C_iw
	\end{equation}
	where $V_i(s)\equiv\left[X_i-\xi\right](s),\forall s\geq0$. In addition, for simplicity of notation, let $V(s)\equiv \left[X-\xi\right](s),\forall s\geq0$ and notice that $V(\cdot)$ is a nonincreasing right-continuous stochastic process such that $V(0)=X(0)$ is a nonnegative square-integrable random variable.
	
	Now, all customers are familiar with the price function $p(\cdot)$ and the statistical assumptions of the model. Furthermore, for every $i\geq1$, the $i$'th customer observes the evolution of $X_i(\cdot)$ from the initiation of his service and until departure. Importantly, since neither balking nor reneging are allowed, there is no strategic interaction between the customers. Thus, if $\mathbb{F}^{X_i}=\left(\mathcal{F}_s^{X_i}\right)_{s\geq0}$ is the natural filtration which is associated with $X_i(\cdot)$, the service duration of the $i$'th customer is a solution of the problem
	\begin{equation} \label{optimization: individual}
	\begin{aligned}
	& \max_{S\in\mathcal{F}}:
	& &w(S)\equiv E\left[\int_0^S X_i(s)ds-p(S)\right] \\
	& \ \text{s.t:}
	& &  S \text{ is a stopping time w.r.t. } \mathbb{F}^{X_i}\,.
	\end{aligned}
	\end{equation}
	Moreover, to break ties, assume that if $S'$ and $S''$ are two stopping times such that $w\left(S'\right)=w\left(S''\right)$, $S'\leq S''$, $P$-a.s. and $P\left(S'<S''\right)>0$, then all customers consider $S'$ as better than $S''$.  
	In particular, notice that $w(\cdot)$ is not indexed by $i$ because $X_1(\cdot),X_2(\cdot),\ldots$ is an iid sequence of processes. 
	
	Eventually, once  \eqref{optimization: individual} has a solution, $S_i$, since $X_1(\cdot),X_2(\cdot),\ldots$ is an iid sequence of processes which is independent from the arrival process, then so is the corresponding sequence of solutions $\left(S_i\right)_{i=1}^\infty$. Thus, the resulting queue is a standard $M/G/1$ queue with a service distribution which is determined endogenously  by the choice of a price function $p(\cdot)$ through the  individual optimization \eqref{optimization: individual}. 
	\begin{remark}
		\normalfont The tie-breaking rule which is given below \eqref{optimization: individual} does not imply ordering between any two stopping times yielding the same objective value. To see this, let $\Omega=\left\{0,1\right\}$, $\mathcal{F}=2^\Omega$ and $P\left(\{0\}\right)=P\left(\{1\}\right)=\frac{1}{2}$. In addition, assume that for every $s\geq0$
		\begin{equation}
		X(s,0)\equiv \begin{dcases}
		2 & 0\leq s< 2 \\
		-2 & 2\leq s \\
		\end{dcases} \ \ , \ \ X(s,1)\equiv\begin{dcases}
		1 & 0\leq s< 3 \\
		-1 & 3\leq s \\
		\end{dcases}
		\end{equation} 
		In particular, notice that for every $\omega\in\Omega$, $s\mapsto X(s,\omega)$ is nonincreasing and right-continuous. Now, for every $\omega\in\Omega$ define 
		\begin{equation}
		S_1(\omega)\equiv \begin{dcases}
		3 & \omega=0 \\
		2 & \omega=1 \\
		\end{dcases}\ \ , \ \ S_2(\omega)\equiv \begin{dcases}
		1 & \omega=0 \\
		4 & \omega=1 \\
		\end{dcases}
		\end{equation}
		and for simplicity assume that $p(\cdot)$ is identically zero. Under these assumptions, it can be verified that $w(S_1)=w(S_2)=2$. In addition, observe that
		\begin{equation}
		\left\{\omega=0\right\}=\left\{X(0)=2\right\}\in\mathcal{F}_0^X
		\end{equation}
		and hence for every $s\geq0$, deduce that $\mathcal{F}_s^X=\mathcal{F}$. Since $S_1,S_2\in\mathcal{F}$, this implies that $S_1$ and $S_2$ are both stopping time with respect to $\mathbb{F}^X$. Finally, observe that $S_1$ and $S_2$ are not ordered since 
		\begin{equation}
		P\left(S_1<S_2\right)=P\left(\left\{1\right\}\right)=\frac{1}{2}>0\,.
		\end{equation}  
		
	\end{remark}	
	\section{The social planner's problem}\label{subsec: social planner}
	In this work the assumption is that the social planner's objective is to find $p(\cdot)$ for which:
	\begin{description}
		\item[(a)] There exists a unique solution of \eqref{optimization: individual} (with the tie-breaking rule).
		
		\item[(b)] The resulting $M/G/1$ queue is stable.
		
		\item[(c)] The mean waiting time of the resulting $M/G/1$ queue is finite.
		
		\item[(d)] The long-run average social welfare which is associated with the resulting $M/G/1$ queue is maximized over the set of all price functions for which all previous three conditions are satisfied. 
	\end{description}
	In particular, if $p(\cdot)$ satisfies these four conditions all together, then it is called \textit{an optimal price function}. The following proposition stems from some standard renewal reward arguments: 
	\begin{proposition}\label{prop: long run}
		Assume that $p(\cdot)$ is a price function for which \eqref{optimization: individual} (with the breaking-tie rule) has a unique solution $S_i$ such that 
		\begin{equation}
		ES_1<\lambda^{-1}\ \ , \ \ ES_1^2<\infty\,.
		\end{equation}
		Then, the long-run average social welfare equals to $\lambda\phi(S_1)$ where
		\begin{equation}\label{eq: long run social}
		\phi(S_1)\equiv E\int_0^{S_1}V(s)ds-\frac{\gamma\lambda ES_1^2}{2\left(1-\lambda ES_1\right)}\,.
		\end{equation}
	\end{proposition}
	
	\begin{proof}
		See Appendix.
	\end{proof}
	\begin{remark}\normalfont
		Note that if $p(\cdot)$ implies multiple solutions of \eqref{optimization: individual} (with the tie-breaking rule), then it is not a good manipulator of individual behaviour, \textit{i.e.,} it is not clear how the customers are acting. In such a scenario, the social planner does not know the resulting service distribution and hence the social value $\phi(\cdot)$ is not well defined. Of-course, some additional assumptions may solve this ambiguity. For example, assuming that the social planner considers the worst-case or uses a prior belief. However, in order to reduce complexities in the model, uniqueness requirement is introduced into \textbf{(a)}. Importantly, under standard regularity conditions the tie-breaking rule implies a unique solution of \eqref{optimization: individual} (see, \textit{e.g.,} Theorem 2.2 in \cite{ Peshkir}).
		
	\end{remark}

	\section{The main result} \label{sec: price function}
	The following Theorem 1 is the main result of this paper.
	\begin{theorem}\label{thm:main result}
		There exist $x^*\in(0,\infty)$ and  $\alpha^*\in\left(0,\frac{EV(0)}{\gamma\lambda}\right]\cap\left(0,\lambda^{-1}\right)$ such that for every $\pi\in\mathbb{R}$, \begin{equation}\label{eq: pi^*}
			p_{\pi,\alpha^*,x^*}(s)\equiv \pi+s x^*+s^2\frac{\gamma\lambda}{2\left(1- \lambda\alpha^*\right)}+\int_0^s\xi(t)dt\ \ , \ \ \forall s\geq0
			\end{equation}
			is an optimal price function. Moreover, when $p_{\pi,\alpha^*,x^*}(\cdot)$ is implemented, then for every $i\geq1$ the service duration of the $i$'th customer equals to
			\begin{equation}
			S^*_i\equiv \inf\left\{s\geq0;V_i(s)-\frac{\gamma\lambda}{1-\lambda\alpha^*}s\leq x^*\right\}\,,
			\end{equation}
			 $ES^*_1=\alpha^*$, $\phi(S_1^*)>0$ and
			\begin{equation}\label{eq: x equation}
			x^*= \gamma\frac{\lambda^2E\left(S^*_1\right)^2}{2\left(1-\lambda E S^*_1\right)^2}\,.
			\end{equation} 
	\end{theorem}
	\begin{equation*}
	\end{equation*}
	
	Observe that an insertion of $\alpha^*=ES^*_1$ and \eqref{eq: x equation} into \eqref{eq: pi^*} implies that for every $\pi\in\mathbb{R}$ (and especially for $\pi=0$) 
	\begin{equation*}
	p^*_\pi(s)=\pi+\gamma\left[s\frac{\lambda^2E\left(S^*_1\right)^2}{2\left(1- \lambda ES^*_1\right)^2}+s^2\frac{\lambda}{2\left(1- \lambda ES^*_1\right)}\right]+\int_0^s\xi(t)dt\ \ , \ \ \forall s\geq0
	\end{equation*}
	is an optimal price function.  In addition, note that  $x^*>0$ and hence for every $\pi\in\mathbb{R}$,  $p^*_\pi(\cdot)$ is an  increasing continuous function such that $p(0)=\pi$. Thus, $\pi$ may be considered as an entry fee for joining the queue. 
	
	\subsection{Internalization of externalities}
	In general, an externality is an economic phenomenon which occurs when an agent has no other choice but being affected by the economic activity of another agent. In the context of the current model, the waiting customers have no choice but being affected by the decision-making of the customer who receives service. 
	Welfare economics  points out that once this interaction is not priced appropriately, then the market will become inefficient. On the other hand, it turns out that once this interaction is priced appropriately, the market failure is corrected and it is common to say that the externality is internalized 
	(see, \textit{e.g.,} Chapter 11 of \cite{MWG1995}).
	
	\ \ \ \ In the context of queueing, \cite{Haviv1998} suggested a measure of the externalities which are caused by a customer with a service demand of $s\geq0$ minutes in a stable $M/G/1$ queue which is operated according to a non-preemptive work-conserving strong service discipline. Specifically, assume that the system is in a steady-state and there is a tagged customer with a service demand of $s\geq0$ minutes. Then, the externalities which are caused by the tagged customer are measured by the total waiting time of other customers that could be saved if the tagged customer gave up on his service demand. Observe that this is a nonnegative random variable which is associated with an interpretation of the loss (in terms of waiting time) which is inflicted on the other customers due to the service demand of the tagged customer. Theorem 2.1 of \cite{Haviv1998} states that once $S$ is a random variable which is distributed like a service time in this $M/G/1$ queue, then the expression of the corresponding expected externalities (as a function of $s$) is given by
	\begin{equation}\label{eq: externalities definition1}
	s\frac{\lambda^2ES^2}{2\left(1- \lambda ES\right)^2}+s^2\frac{\lambda}{2\left(1- \lambda ES\right)}\,.
	\end{equation}

 Now, in the current model, $\int_0^s\xi(t)dt$ is the (deterministic) loss which is inflicted on the server by the tagged customer due to a service demand of $s\geq0$ minutes. Hence,  $p_0^*(\cdot)$ is an optimal price function which internalizes the externalities. Namely, this price function is such that every customer pays the expected loss which is inflicted on the society (server and other customers all together) by his service demand when all other customers are acting according to their self-interest. Another observation is that for every $\pi\in\mathbb{R}$, $p^*_\pi(\cdot)$ internalizes the marginal externalities, \textit{i.e.,} $\frac{dp^*_\pi}{ds}$ coincides with the marginal externalities.
 
\subsection{Suboptimality of $p^*_\pi(\cdot)$ in a wider sense} \label{subsec: tax function}
	Observe that the optimality of $p^*_\pi(\cdot)$ is attained with respect to the class of price functions which are determined uniquely by the service duration. It turns out that all price functions which belong to this class are suboptimal when considering a larger optimization domain. To show this, for simplicity assume that $\xi\equiv0$ and $V(s)$ is a positive nonincreasing and right-continuous function such that $V(s)\rightarrow0$ as $s\rightarrow\infty$. In addition, enlarge the optimization domain by letting the social planner setting a price function which is determined uniquely by the queue length and service duration. Since $V(s)>0$ for every $s\geq0$, an optimal price function in this new setup must allow a customer to get service for free once there are no waiting customers. The only function which is uniquely determined by service duration and satisfies this property is constant.  Assume by contradiction that a constant function is optimal in the new setup. Since it is uniquely determined by service duration, then Theorem \ref{thm:main result} implies that by setting a price function $p^*_\pi(\cdot)$ the social planner attains the same objective value. Thus,  $p^*_\pi$ is an optimal price function in the new setup. On the other hand, Theorem \ref{thm:main result} states that $p_\pi^*(\cdot)$ is not a constant function which implies a contradiction. 
	
	Whilst this is a drawback of the current analysis, as mentioned by \cite{Haviv1998}, there are some arguments for a price function which is uniquely determined by service duration:
\begin{enumerate}
	\item In some occasions monitoring the queue length is not possible due to technical reasons. Also, it is possible that such monitoring is too expensive. 
	
	\item It might be a bit unfair to require different payments from customers having the same service duration. For instance, a customer may refuse to pay large amounts of money due to a batch of arrivals happened just after he had started receiving service claiming that it is not her fault. 
\end{enumerate} 

\section{Proof} \label{sec: proof}
Generally speaking, the proof of Theorem \ref{thm:main result} may be divided into two stages:\newline
\begin{description}
	\item[Stage 1: Optimization] \begin{equation*}
	\end{equation*}
	Solve:
	\begin{equation}\label{optimization: main1}
	\max_{S\in\mathcal{D}}:\phi(S)
	\end{equation}
	where
	\begin{equation}
	\mathcal{D}\equiv\left\{S\in\mathcal{F};S\geq0\,\ ES<\lambda^{-1}\ ,\ ES^2<\infty\right\}
	\end{equation}
	and denote a maximizer by $S^*$.
	\newline
	\item[Stage 2: Reverse engineering]
	\begin{equation*}
	\end{equation*} 
	Find $p(\cdot)$ for which  \eqref{optimization: individual} has a unique solution which has the same distribution like $S^*$.\newline
\end{description}
	\subsection{Optimization}
	In order to solve \eqref{optimization: main1}, consider the following optimization: 
	\begin{equation} \label{optimization: social1}
	\begin{aligned}
	& \max_{S\in\mathcal{F}}:
	& &f(S)\equiv E\int_0^S \left[V(s)-\frac{\gamma\lambda s}{1-\lambda ES}\right]ds \\
	& \ \text{s.t:}
	& & 0\leq S \ , \ P \text{-a.s.}\,, \\ & & & ES<\lambda^{-1}\
	\end{aligned}
	\end{equation}
	 and denote the domain of \eqref{optimization: social1} by $\mathcal{S}$. Notice that $EV^2(0)<\infty$ implies that 
	\begin{equation} \label{eq: regularity condition 0}
	E\int_0^\infty\left[\ V(s)-\gamma \lambda s\right]^+ds\leq E\int_0^\infty\left[V(0)-\gamma\lambda s\right]^+ds<\infty
	\end{equation}
	and hence
	\begin{equation}\label{eq: regularity condition}
	E\int_0^\infty\left[V(s)-\frac{\gamma\lambda s}{1-\lambda\alpha}\right]^+ds<\infty\ \ , \ \ \forall\alpha\in\left[0,\lambda^{-1}\right)\,.
	\end{equation}  
	This shows that the objective functional $f(\cdot)$ is well-defined on $\mathcal{S}$. 
	
	Then, the purpose is to find a random variable $S^*$ which is a maximizer of \eqref{optimization: social1} such that $S^*\in\mathcal{D}$. Since $f(\cdot)$ is an extension of $\phi(\cdot)$, deduce that $S^*$ is also a maximizer of \eqref{optimization: main1}.
	In order to solve \eqref{optimization: social1}, a two phase method is suggested.
	\subsubsection{Phase I:}
	For every $\alpha\in\left[0,\lambda^{-1}\right)$ 
	and $x\in\left[-\infty,\infty\right]$ define
	\begin{equation}
	S_\alpha(x)\equiv\inf\left\{s\geq0; V(s)-\frac{\gamma\lambda s}{1-\lambda\alpha}\leq x\right\}\,.
	\end{equation} 
	Let $\alpha\in\left[0,\lambda^{-1}\right)$ and the statement of Phase I is given by \eqref{optimization: social1} with an additional constraint $ES=\alpha$, \textit{i.e.,} 
	\begin{equation} \label{optimization: phase I}
	\begin{aligned}
	& \max_{S\in\mathcal{F}}:
	& &E\int_0^S\left[ V(s)-\frac{\gamma\lambda s}{1- \lambda\alpha}\right]ds \\
	& \ \text{s.t:}
	& & 0\leq S \ , \ P \text{-a.s.}\,, \\ & & & ES=\alpha\,. 
	\end{aligned}
	\end{equation}
	If $\alpha=0$, then $S_0(\infty)=0$ and hence it is a solution of Phase I with $\alpha=0$. Assume that $\alpha\in\left(0,\lambda^{-1}\right)$ and
	observe that $ V(s)-\frac{ \gamma s}{1- \lambda\alpha}$ is decreasing w.r.t  $s$. In addition, since $ V(\cdot)$ is nonincreasing, then for every $x\in\mathbb{R}$
	\begin{equation}\label{eq: bound on zeta}
	0\leq S_\alpha(x)\leq\frac{V(0)+|x|}{\gamma\lambda}\,.
	\end{equation}
	It is given that $EV^2(0)<\infty$ which implies that $ES_\alpha(x)<\infty$ for every $x\in\mathbb{R}$. Thus, by Theorem 1 of \cite{Jacobovic2020} (see also the last paragraph before Proposition 1 of the same reference) there exists $x_\alpha\in\mathbb{R}$ for which $S_\alpha\equiv S_\alpha(x_\alpha)$ is a solution of \eqref{optimization: phase I}. 
	
	\subsubsection{Phase II:}\label{subsec: phase II original}
	For every $\alpha\in\left[0,\lambda^{-1}\right)$, define 
	
	\begin{equation}
	g(\alpha)\equiv f(S_\alpha)= E\int_0^{S_\alpha}\left[V(s)-\frac{\gamma\lambda}{1-\lambda\alpha}s\right]ds
	\end{equation}
	and the statement of Phase II is given by 
	\begin{equation}
	\max_{0\leq\alpha<\lambda^{-1}}g(\alpha)\,.
	\end{equation}
	The main result which follows from the analysis of this phase is summarized by the following Proposition \ref{lemma: existence}. 
	\begin{proposition}\label{lemma: existence}
		$g(\cdot)$ is a right-continuous concave function which is maximized at a point $\alpha^*\in\left(0,\frac{EV(0)}{\gamma\lambda}\right]\cap\left(0,\lambda^{-1}\right)$. In addition, $S^*\equiv S_{\alpha^*}$ satisfies $E(S^*)^2<\infty$, $f(S^*)>0$ and
		\begin{equation}
		x^*\equiv x_{\alpha^*}= \gamma\frac{\lambda^2E\left(S^*\right)^2}{2\left(1-\lambda E S^*\right)^2}\in(0,\infty)\,.
		\end{equation} 
		\end{proposition}
	\begin{proof}
		See Appendix.
	\end{proof}
	
	\subsection{Reverse engineering}
	
	Fix $\pi\in\mathbb{R}$ and let $\alpha^*$ be the same like it was in the statement of Proposition \ref{lemma: existence}. Then, the next step is to show that $p_\pi^*\equiv p_{\pi,\alpha^*,x_{\alpha^*}}$ is an optimal price function. This is to be done by showing that when $p^*_\pi$ is implemented, then each of the individual optimizations of the customers has a unique solution which is distributed like $S^*$ given in Proposition \ref{lemma: existence}. 
	
	In practice, let $i\geq1$ and notice that when the social planner sets a price function $p_\pi^*(\cdot)$, then the resulting marginal price of service is given by
		\begin{equation*}
		Mp^*(s)\equiv x_{\alpha^*}+\frac{\gamma\lambda s}{1-\lambda\alpha^*}+\xi(s) \ \ , \ \ \forall s>0
		\end{equation*}
		which is a nondecreasing right-continuous function of $s$. Thus, since $X_i(\cdot)$ is nonincreasing and right-continuous, then the  $i$'th customer solves \eqref{optimization: individual} by a departure at the first moment when his marginal utility is not greater than the marginal price. This means that the service duration of the $i$'th customer equals to
		\begin{equation*}
		\inf\bigg\{s\geq0;X_i(s)\leq Mp^*(s) \bigg\}=\inf\left\{s\geq0;V_i(s)-\frac{\gamma\lambda s}{1-\lambda\alpha^*}\leq x_{\alpha^*}\right\}= S^*_i\,.
		\end{equation*}
		Observe that $S^*_i$ is a stopping time with respect to $\mathbb{F}^{X_i}$. In addition, $V(\cdot)$ and $V_i(\cdot)$ are equally distributed. Therefore,  $S^*$ and $S^*_i$ are equally distributed and the optimality of $p^*_\pi$ follows.

	\begin{remark}\label{remark: price function p a x}
		\normalfont Observe that the arguments which were made in this subsection lead to the conclusion that for every $\alpha\in\left[0,\lambda^{-1}\right)$ and $x\in\mathbb{R}$
		\begin{equation}
		p_{\alpha,x}(s)\equiv sx+s^2\frac{\gamma\lambda}{2\left(1-\lambda\alpha\right)}+\int_0^s\xi(s)ds\ \ , \ \ \forall s\geq0
		\end{equation}
		is a price function for which there is a unique solution of \eqref{optimization: individual}. In addition, this solution has the same distribution like $S_\alpha(x)$. 
	\end{remark}

	\begin{remark}
		\normalfont The last paragraph in page 4 of \cite{Jacobovic2020} implies that once $\left(\Omega,\mathcal{F},P\right)$ is complete, then the results of this section are valid with a weaker assumption that $V(\cdot)$ is nonincreasing and right-continuous $P$-a.s.   
	\end{remark}
	
	\section{Numerical procedure}\label{sec: numerical procedure}
	Theorem \ref{thm:main result} provides a characterization of an optimal price function up to the values of $\alpha^*$ and $x^*$. This section is about a numerical computation of $\alpha^*$ and $x^*$. In order to carry out this computation, consider an additional assumption that  $V(0)\leq\kappa$, $P$-a.s. for some positive constant $\kappa<\infty$. 
	
	\subsection{Numerical derivation of $\alpha^*$ and $x^*$}
	For every $\alpha\in(0,\lambda^{-1})$, $x_\alpha$ is determined as a solution of the equation $ES_\alpha(x)=\alpha$ in $x$. In addition, note that $x\mapsto S_\alpha(x)$ is nonincreasing and hence $x\mapsto ES_\alpha(x)$ is nonincreasing. Thus, since $ES_\alpha(\kappa)=0$, then $x_0=\kappa$ and $x_\alpha<\kappa$ for every $\alpha\in(0,\lambda^{-1})$. 
	
	Now, assume that there is a program which has an input $(\alpha,x)$ and 
	returns the numerical value of $ES_\alpha(x)$. In particular, it is possible to compute $ES_\alpha(0)-\alpha$ and see whether it is negative. Then, given the result of this query, determine  whether $x_\alpha\geq0$. Since $\alpha\mapsto x_\alpha$ is nonincreasing and Theorem \ref{thm:main result} states that $x^*\in(0,\infty)$,  then $x_\alpha\leq0$ implies that $\alpha>\alpha^*$. Otherwise, $x_\alpha$ may be derived by a standard line search algorithm on $\left[0,\kappa\right]$.
	
	Proposition \ref{lemma: existence} states that $g(\cdot)$ is concave with a maximizer $\alpha^*$. In addition, assume that there is a program which receives $(\alpha,x_\alpha^+)$ and
	\begin{enumerate}
		\item If $x_\alpha^+=0$, then it returns a statement that $\alpha^*<\alpha$.
		
		\item Otherwise, it returns the value of $g(\alpha)$. 
	\end{enumerate}
	Thus, with such a program in hands, a numerical derivation of $\alpha^*$ is possible by standard techniques. This approach will be demonstrated in several special cases in the upcoming subsections.
	
	\subsection{A motivating model}\label{subsec: DSR}
	The purpose of this subsection is to describe a specific model which helps with the interpretation of the examples in the upcoming subsections.
	 
	Consider a channel of communication with a Poisson stream of users. Every user expects to receive a message which is delivered through the channel. Thus, in order to receive the message, the user has to be connected to the channel. It is impossible to have several users who are connected to this channel simultaneously. In addition, the queue for this channel is operated according to a service discipline which is work-conserving and non-preemptive such that reneging and balking are forbidden. 
	
	Assume that the information is transferred through the channel in a constant rate which equals one. In addition, the users have only partial information regarding the message which they are going to receive. Then, every user who is connected dynamically reads the message which is transferred and decides when to disconnect. This motivates the model which is presented in Section \ref{sec: model description}. Namely, let $C_i\geq0$ be the loss rate of the $i$'th customer due to waiting and  $V_i(s)$, $s\in[0,\infty)$ is the infinitesimal value of the $s$'th byte in the message of the  $i$'th user.
	  
	In particular, notice that the nonincreasing assumption fits to scenarios in which the messages are ordered according to the marginal value of information. Namely, in each message, the most valuable bytes are placed at the beginning. Finally, note that for every $i\geq1$, the random variable
	\begin{equation}
	T_i\equiv\inf\left\{s\geq0;V_i(s)\leq0\right\}
	\end{equation}
	may be interpreted as the size of the message which is delivered to the $i$'th user.
	
	\subsection{Constant marginal value}\label{subsec: constant DSR}
	Let $T$ be a nonnegative random variable with a cumulative distribution function (cdf) $F(\cdot)$ such that $P(T>0)>0$. Assume that $V(s)=\kappa1_{[0,T]}(s)$ for every $s\geq0$ where $\kappa\in(0,\infty)$. In such a case, for every $\alpha\in\left(0,\lambda^{-1}\wedge ET\right)$ and $x\in(0,\kappa)$, 
	\begin{align}
	S_\alpha(x)&=\inf\left\{s\geq0;\kappa1_{[0,T]}(s)-\frac{\gamma\lambda s}{1-\lambda\alpha}\leq x\right\}=T\wedge z(\alpha,x)
	\end{align}
	where
	\begin{equation}
	z(\alpha,x)\equiv \frac{(\kappa-x)(1-\lambda\alpha)}{\gamma\lambda}\,.
	\end{equation}
	Therefore, if $F(\cdot)$ is the cumulative distribution function of $T$, then 
	\begin{equation}
	ES_\alpha(x)=\int_0^{z(\alpha,x)}\left[1-F(s)\right]ds\,.
	\end{equation}
	This makes a numerical computation of $x_\alpha$ possible. For simplicity of notation, denote $z_\alpha\equiv z(\alpha,x_\alpha)$ which is the solution of the equation
	\begin{equation}
	\alpha=\int_0^z\left[1-F(s)\right]ds
	\end{equation} 
	in $z$. Then, observe that
	\begin{align}
	ES^2_\alpha(x)&=\int_0^{z_\alpha^2}\left[1-F\left(\sqrt{s}\right)\right]ds\\&=2\int_0^{z_\alpha}u\left[1-F(u)\right]du\,.
	\end{align}
	Furthermore, note that
	\begin{equation}
	E\int_0^{S_\alpha}\kappa1_{[0,T]}(s)ds=\kappa E\left(T\wedge S_\alpha\right)=\kappa ES_\alpha=\kappa\alpha\,.
	\end{equation}
	Thus, for every $\alpha\in\left(0,\lambda^{-1}\wedge ET\right)$, 
	\begin{align}
	g(\alpha)&=\kappa\alpha -\frac{\gamma\lambda}{1-\lambda\alpha}\int_0^{z_\alpha}u\left[1-F(u)\right]du\,.
	\end{align} 
	Thus, it is also possible to compute the value of $\alpha^*$ by standard numerical techniques.  
	
	\subsubsection{$T\sim\exp(q)$}
	Assume that $T\sim\exp(q)$ for some $q\in(0,\infty)$. Then, for every $\alpha\in\left(0,\lambda^{-1}\wedge q^{-1}\right)$ and $x\in(0,\kappa)$
		\begin{equation}
		ES_\alpha(x)=\int_0^{z(\alpha,x)}e^{-qs}ds=q^{-1}\left[1-e^{-qz(\alpha,x)}\right]\,.
		\end{equation}
		Thus, $z_\alpha$ is given by 
		\begin{equation}
		z_\alpha=-\frac{\ln\left(1-\alpha q\right)}{q}
		\end{equation}
		and hence
		\begin{align}
		g(\alpha)&=\kappa\alpha -\frac{\gamma\lambda}{1-\lambda\alpha}\int_0^{-\frac{\ln\left(1-\alpha q\right)}{q}}ue^{-qu}du\\&=\kappa\alpha-\frac{\gamma\lambda}{1-\lambda\alpha}\cdot\frac{(1-\alpha q)\ln(1-\alpha q)+\alpha q}{q^2}\,.\nonumber
		\end{align} 
		Then, by Proposition \ref{lemma: existence}, $g(\cdot)$ is concave and hence, all which is left to do is to maximize it on $\left[0,\lambda^{-1}\wedge q^{-1}\right]$ by using standard numerical techniques. 
		
	\subsection{A marginal value which is linear in the remaining message size} \label{subsec: DSR linear}
	Let $T$ be a nonnegative random variable with a cumulative distribution function (cdf) $F(\cdot)$ such that $P(T>0)>0$.  In addition, let $\kappa\in(0,\infty)$ and assume that $T\leq\kappa$. Then, consider the case where $V(s)=(T-s)^+$ for every $s\geq0$.  Importantly, as explained in the proof of Proposition \ref{lemma: existence}, for simplicity and without loss of generality, it is possible to solve the optimization with $V(s)=T-s$ for every $s\geq0$. In addition, for every $\alpha\in\left(0,\lambda^{-1}\wedge ET\right)$ denote $b(\alpha)\equiv b(\alpha;\lambda)\equiv1+\frac{\gamma\lambda}{1-\lambda\alpha}$ and notice that  
	\begin{equation}\label{eq: S_alpha}
	S_\alpha(x)=\frac{(T-x)^+}{b(\alpha)}\ ,\ \forall\alpha\in\left(0,\lambda^{-1}\wedge ET\right),x\in(0,\kappa)\,.
	\end{equation} 
	Thus, for every $\alpha\in\left(0,\lambda^{-1}\wedge ET\right)$ and $x\in(0,\kappa)$ deduce that
	\begin{equation}\label{eq: x_alpha computation}
	ES_{\alpha}(x)=\frac{1}{b(\alpha)}\int_0^\infty(u-x)^+dF(u)
	\end{equation}
	where $F(\cdot)$ is the cumulative distribution function of $T$. This formula might be used in order to derive $x_\alpha$ by a standard line-search.
	
	Now, after explaining how to compute $x_\alpha$ for every  $\alpha\in\left(0,\lambda^{-1}\wedge ET\right)$, the next step is to show how to compute $g(\alpha)$. To this end, notice that 
	\begin{align}\label{eq: g formula}
	g(\alpha)&=E\int_0^{S_\alpha}\left[T-b(\alpha)s\right]ds\\&=ET S_\alpha-\frac{b(\alpha)}{2}ES_\alpha^2\,.\nonumber
	\end{align}   
	Now, by \eqref{eq: S_alpha} with an insertion of $x=x_\alpha$, deduce that 
	\begin{align} \label{eq: joint moment}
	ET S_\alpha&=\frac{E (T-x_\alpha+x_\alpha)\left(T-x_\alpha\right)1_{(x_\alpha,\infty)}(T)}{b(\alpha)}\\&=\frac{E\left[\left(T-x_\alpha\right)^+\right]^2+x_\alpha E\left(T-x_\alpha\right)^+}{b(\alpha)}\nonumber\\&=b(\alpha)ES_\alpha^2+x_\alpha\alpha\nonumber
	\end{align}
	and 
	\begin{equation}
	ES_\alpha^2=\frac{E\left[\left(T-x_\alpha\right)^+\right]^2}{b^2(\alpha)}\,.
	\end{equation}
	Clearly, this allows a computation of $g(\alpha)$ for every $\alpha\in\left(0,\lambda^{-1}\wedge ET\right)$. Finally, since $g(\cdot)$ is concave on $\left(0,\lambda^{-1}\wedge ET\right)$, then in order to derive $\alpha^*$, it is possible to use a line search procedure on the interval $\left(0,\lambda^{-1}\wedge ET\right)$. 
	
	\subsubsection{$T\sim U\left[t,\kappa\right]$}\label{subsec: linear DSR uniform} 
	Consider the case when $T\sim U[t,\kappa]$ for some $t\in[0,\kappa)$. Then, for every $\alpha\in\left(0,\lambda^{-1}\wedge\frac{\kappa+t}{2}\right)$ 
	\begin{equation}
	ES_\alpha(x)= \frac{1}{b(\alpha)}\begin{dcases}
	\frac{(\kappa-x)^2}{2(\kappa-t)} & t<x<\kappa \\
	\frac{t+\kappa}{2}-x & 0\leq x\leq t 
	\end{dcases}\,.
	\end{equation}
	Now, for simplicity of notation, denote $l(\alpha)=\alpha b(\alpha),\forall \alpha\in\left[0,\lambda^{-1}\right)$ and observe that it is an increasing function onto $[0,\infty)$. Then, deduce that for every $\alpha\in\left(0,\lambda^{-1}\wedge\frac{\kappa+t}{2}\right)$,
	\begin{equation}
	x_\alpha=\begin{dcases}
	\kappa-\sqrt{2(\kappa-t)l(\alpha)} & 0\leq l(\alpha)<\frac{\kappa-t}{2} \\
	\frac{t+\kappa}{2}-l(\alpha) & \frac{\kappa-t}{2}\leq l(\alpha)\leq\frac{\kappa+t}{2} 
	\end{dcases}\,.
	\end{equation} 
	In addition, by similar arguments, it can be observed that:
	\begin{align}
	ES_\alpha^2&=\frac{1}{b^2(\alpha)}\begin{dcases}
	\frac{\kappa-x_\alpha}{\kappa-t}\cdot\frac{(\kappa-x_\alpha)^2}{3} & 0\leq l(\alpha)<\frac{\kappa-t}{2} \\
	\frac{(k-x_\alpha)^3-(t-x_\alpha)^3}{3(k-t)} & \frac{\kappa-t}{2}\leq l(\alpha)\leq\frac{\kappa+t}{2} 
	\end{dcases}\,.
	\end{align}
	Then, an insertion of this expression into \eqref{eq: g formula} and \eqref{eq: joint moment} implies a closed form expression. By Proposition \ref{lemma: existence}, this function is concave and hence may be optimized by standard numerical techniques.

	\subsection{A marginal value which is a constant minus a subordinator}\label{subsec: subordinator}
	Let $\mathbb{F}$ be some filtration of $\mathcal{F}$ which is augmented and right-continuous. Then, assume that $\left\{J(s);s\geq0\right\}$ is a subordinator, \textit{i.e.}, a nondecreasing, right-continuous process with stationary and independent increments with respect to $\mathbb{F}$ such that $J(0)=0$, $P $-a.s. It is known that $Ee^{-J(s)t}=e^{-\eta(t)s}$ for every $t,s\geq0$ where
	\begin{equation}
	\eta(t)\equiv ct+\int_{(0,\infty)}\left(1-e^{-t x}\right)\nu(dx)\ \ , \ \ \forall t\geq0\,,
	\end{equation} 
	$c\geq0$ and $\nu$ is the associated L\'evy measure which satisfies
	\begin{equation}
	\int_{(0,\infty)}\left(x\wedge 1\right)\nu(dx)<\infty\,.
	\end{equation}
	In particular, $\eta(\cdot)$ is referred as the exponent of $J(\cdot)$. In addition, denote 
	\begin{equation}
	\rho\equiv EJ(1)=\eta'(0)=c+\int_{(0,\infty)}x\nu(dx)=c+\int_0^\infty\nu\left[(x,\infty)\right]dx
	\end{equation}
	and assume that $\rho\in(0,\infty)$.
	
	Let $\kappa\in(0,\infty)$ be some constant and consider a process $V(s)=\kappa-J(s)$ for every $s\geq0$. In particular, this is a nonincreasing jump process with a nonpositive drift. 
	
	Now, for every $\alpha\in\left[0,\lambda^{-1}\right)$ and $s\geq0$, define $J_\alpha(s)\equiv J(s)+\frac{\gamma\lambda}{1-\lambda\alpha}s$ which is a subordinator with L\'evy measure $\nu$,  parameter $c_\alpha\equiv c+\frac{\gamma\lambda}{1-\lambda\alpha}$ and exponent $\eta_\alpha(\cdot)$. Then, given $\alpha\in\left[0,\lambda^{-1}\right)$ and $x\in\left[0,\kappa\right]$, observe that
	\begin{align}\label{eq: x_alpha equation}
	S_\alpha(x)&=\inf\left\{s\geq0;J(s)\geq\kappa-x-\frac{ \gamma\lambda}{1-\lambda\alpha}s\right\}\nonumber\\&=\inf\left\{s\geq0;J_\alpha(s)\geq\kappa-x\right\}\\&=\inf\left\{s\geq0;J_\alpha(s)>\kappa-x\right\}\nonumber
	\end{align}
	where the last equality holds because $J_\alpha(\cdot)$ is an increasing process. To derive a relevant formula in terms of potential measures, it is known (see, \textit{e.g.}, Equation (8) in \cite{Alili2005}) that  
	\begin{equation} \label{eq: LST exit}
	Ee^{-t J_\alpha\left[S_\alpha(x)\right]}=\eta_\alpha(t)\int_{\kappa-x}^\infty e^{-t z}U_\alpha(dz)\ \ , \ \ \forall t\geq0
	\end{equation} 
	where $U_\alpha(\cdot)$ is a potential measure which is defined via
	\begin{equation}
	\int_0^\infty e^{-t z}U_\alpha(dz)=\frac{1}{\eta_\alpha(t)}\ \ , \ \ \forall t\geq0\,.
	\end{equation}
	By these equations, differentiating \eqref{eq: LST exit} w.r.t. $t$ and taking $t\downarrow0$ deduce that 
	\begin{align} \label{eq: EJ(X_e)}
	EJ_\alpha\left[S_\alpha(x)\right]=\eta'_\alpha(0)\int_0^{\kappa-x} U_\alpha(dz)\,.
	\end{align}
	Thus, by plugging this result into Equation 3.7 of \cite{Bekker2008}, deduce that 
	\begin{equation} \label{eq: EX_e subordinator}
	ES_\alpha(x)=\frac{EJ_\alpha\left[S_\alpha(x)\right]}{\eta_\alpha'(0)}=\int_0^{\kappa-x} U_\alpha(dz)\,.
	\end{equation}
	This formula might be used in order to derive $x_{\alpha'}$ by a standard line-search procedure on $\left[0,\kappa\right]$. 
	
	Then, it is left to to develop a formula of $g(\alpha)$. To this end, let $S_\alpha=S_\alpha\left(x_\alpha\right)$ and notice that 
	\begin{equation} \label{eq: subordinator 2}
	EJ^2_\alpha\left(S_\alpha\right)=2\eta_\alpha'(0)\int_0^{\kappa-x_\alpha}zU_\alpha(dz)-\eta_\alpha''(0)\int_0^{\kappa-x_\alpha}U_\alpha(dz)
	\end{equation}
	can be derived by a similar fashion to \eqref{eq: EJ(X_e)}. In addition, for every $t\geq 0$, the  Kella-Whitt martingale (see Theorem 2 of \cite{Kella1992}) which is associated with $J_\alpha(\cdot)$ is given by
	\begin{equation}
	M_\alpha(s;t)\equiv-\eta_\alpha(t)\int_0^se^{-t J_\alpha(s)}ds+1-e^{-t J_\alpha(s)}\ \ , \ \ \forall s\geq0\,. 
	\end{equation}
	It is known that this is a zero-mean martingale. Thus, by applying Doob's optional stopping theorem w.r.t. $S_\alpha\wedge s$ for some $s>0$ and then taking $s\to\infty$ using monotone and bounded convergence theorems, deduce that
	\begin{equation}
	E\int_0^{S_\alpha}e^{-t J_\alpha(s)}ds=\frac{1-Ee^{-t J_\alpha(S_\alpha)}}{\eta_\alpha(t)}\ \ , \ \ \forall t>0\,.
	\end{equation}   
	Now, by differentiating both sides w.r.t. $t$, for every $t>0$ obtain
	\begin{equation}
	E\int_0^{S_\alpha}J_\alpha(s)e^{-t J_\alpha(s)}ds=\frac{\eta_\alpha'(t)\left[1-Ee^{-t J_\alpha(S_\alpha)}\right]-\eta_\alpha(t)EJ_\alpha(S_\alpha)e^{-t J_\alpha(S_\alpha)}}{\eta_\alpha^2(t)}\,.
	\end{equation}
	Thus, by taking a limit $t\downarrow0$ using monotone convergence with the help of l'Hopital's rule (twice) deduce that 
	\begin{equation}\label{eq: integral}
	E\int_0^{S_\alpha}J_\alpha(s)ds=\frac{\eta_\alpha'(0)EJ_\alpha^2(S_\alpha)+\eta_\alpha''(0)EJ_\alpha(S_\alpha)}{2\left[\eta_\alpha'(0)\right]^2}\,.
	\end{equation}
	Now, observe that this result can be plugged into the objective function of Phase II, \textit{i.e.}, 
	\begin{align} \label{eq: g(alpha) subordinator}
	g(\alpha)&=E\int_0^{S_\alpha}\left[\kappa-J_\alpha(s)\right]ds\\&=\kappa\alpha-\frac{\eta_\alpha'(0)EJ_\alpha^2(S_\alpha)+\eta_\alpha''(0)EJ_\alpha(S_\alpha)}{2\left[\eta_\alpha'(0)\right]^2}\,.\nonumber
	\end{align}
	Thus, by an insertion of \eqref{eq: EJ(X_e)} and \eqref{eq: subordinator 2} into \eqref{eq: g(alpha) subordinator}, derive an expression of $g(\alpha)$ in terms of integrals with respect to $U_\alpha(\cdot)$.
	Finally, by Proposition \ref{lemma: existence}, $g(\cdot)$ is concave on $\left[0,\lambda^{-1}\right)$ and hence standard numerical techniques might be applied in order to maximize it.

	\subsubsection{When $J(\cdot)$ is a Poisson process}\label{subsubsec: Poisson}
	Assume that $J(\cdot)$ is a Poisson process with rate $q\in(0,\infty)$. Let $\alpha\in\left[0,\lambda^{-1}\right)$, $x\in\left[0,\kappa\right]$ and for every $j=0,1,\ldots,\lfloor \kappa-x\rfloor$ denote  
	\begin{equation}
	s_j\equiv\frac{(1-\lambda\alpha)(\kappa-x-j)}{\gamma\lambda}\,.
	\end{equation}
	In addition, let $s_{\lfloor \kappa-x\rfloor+1}\equiv0$. Especially, notice that 
	\begin{equation}
	\kappa-x-\frac{ \gamma\lambda}{1-\lambda\alpha}s_j=j\ \ , \ \ \forall j=0,1,\ldots,\lfloor \kappa-x\rfloor
	\end{equation}
	and for every $s\geq0$ define
	\begin{equation}
	\delta(s)\equiv\lfloor k-x-\frac{ \gamma\lambda}{1-\lambda\alpha}s\rfloor\,.
	\end{equation}
	Then, observe that for every $x\in\left[0,\kappa\right]$ and $s\in\left[0,\infty\right)\setminus\left\{s_0,s_1,\ldots,s_{\lfloor k-x\rfloor+1}\right\}$ 
	\begin{align}
	P\left[S_\alpha(x)>s\right]&=P\left[J(s)<k-x-\frac{ \gamma\lambda}{1-\lambda\alpha}s\right]\\&=1_{[0,\infty)}\left(k-x-\frac{ \gamma\lambda}{1-\lambda\alpha}s\right)\sum_{n=0}^{\delta(s)}e^{-qs}\frac{(qs)^n}{n!}
	\end{align}
	Thus, 
	\begin{align}\label{eq: z_alpha formula}
	ES_\alpha(x)&=\int_0^\infty1_{[0,\infty)}\left(k-x-\frac{ \gamma\lambda}{1-\lambda\alpha}s\right)\sum_{n=0}^{\delta(s)}e^{-qs}\frac{(qs)^n}{n!}ds\\&\nonumber=q^{-1}\sum_{j=0}^{\lfloor\kappa-x\rfloor}\sum_{n=0}^j\int_{s_{j+1}}^{s_j}e^{-qs}\frac{q^{n+1}s^n}{n!}ds\\&=\nonumber\sum_{j=0}^{\lfloor\kappa-x\rfloor}\sum_{n=0}^j\sum_{m=0}^n\frac{q^{m-1}}{m!}\left(e^{-q s_{j+1}}s_{j+1}^m-e^{-q s_j}s_j^m\right)
	\end{align}
	is a formula that can be used in order to find $x_\alpha$. In a similar fashion, for every $\alpha\in\left[0,\lambda^{-1}\right)$ and $s\geq0$  
	
	\begin{equation}
	P\left(S_\alpha^2>s\right)=P\left[S_\alpha(x_\alpha)>\sqrt{s}\right]
	\end{equation} 
	and hence 
	\begin{align}
	ES_\alpha^2&=\int_0^\infty P\left[S_\alpha(x_\alpha)>\sqrt{s}\right]ds\\&=\sum_{j=0}^{\lfloor k-x_\alpha\rfloor}\sum_{n=0}^j\int_{s_{j+1}}^{s_j}e^{-q\sqrt{s}}\frac{q^n\left(\sqrt{s}\right)^n}{n!}ds\\&=\sum_{j=0}^{\lfloor k-x_\alpha\rfloor}\sum_{n=0}^j\frac{2(n+1)}{q^2}\int_{\sqrt{s_{j+1}}}^{\sqrt{s_j}}e^{-qy}\frac{q^{n+2}y^{n+1}}{(n+1)!}dy\\&=\sum_{j=0}^{\lfloor k-x_\alpha\rfloor}\sum_{n=0}^j 2(n+1)\sum_{m=0}^{n+1}\frac{q^{m-2}}{m!}\left(e^{-q \sqrt{s_{j+1}}}s_{j+1}^{\frac{m}{2}}-e^{-q \sqrt{s_j}}s_j^\frac{m}{2}\right)\,.
	\end{align} 
	Therefore, with the help of \eqref{eq: integral} (Note that this equation remains valid when $J_\alpha(\cdot)$ and $\eta_\alpha(\cdot)$ are replaced by any other subordinator and its exponent). This implies that
	\begin{align}\label{eq: g(alpha) poisson}
	g(\alpha)&=\kappa\alpha-E\int_0^{S_\alpha}J(s)ds-\frac{\gamma\lambda ES_\alpha^2}{2\left(1-\lambda\alpha\right)^2}\\&=\kappa\alpha-\frac{EJ^2(S_\alpha)+EJ(S_\alpha)}{2q}-\frac{\gamma\lambda ES_\alpha^2}{2\left(1-\lambda\alpha\right)^2}\,.\nonumber
	\end{align}
	Now, observe that $J(S_\alpha)$ is a discrete random variable with support $\mathcal{N}\equiv\left\{0,1,\ldots,\lfloor\kappa-x_\alpha\rfloor+1\right\}$. Thus, since $J(\cdot)$ maintains independent increments, then for every $n\in\mathcal{N}$
	\begin{align}
	P\left[J(S_\alpha)=n\right]&=P\left[J(s_n)=n\right]\\&+\sum_{i=0}^{n-1}P\left[J(s_n)=i\right]P\left[J(s_n)-J(s_{n-1})\geq n-i\right]\,.\nonumber
	\end{align}
	Finally, note that $J(s_n)\sim\text{Poi}\left(qs_n\right)$ and $J(s_n)-J(s_{n-1})\sim\text{Poi}\left(q\frac{1-\lambda\alpha}{ \gamma\lambda}\right)$. Therefore, all of these probabilities have closed form expressions and so are $EJ(S_\alpha)$ and $EJ^2(S_\alpha)$.
	
	\begin{remark}\normalfont
		This subsection is closely related to the crossing time of a Poisson process by a decreasing linear boundary. For more information regarding this issue with some others related topics, see \cite{Zacks2017}.
	\end{remark}
	
	\subsection{A marginal value which is a nonincreasing MMFF}
	Assume that $V(\cdot)$ is a nonincreasing Markov modulated fluid flow (MMFF). More precisely, let $\left\{J(t);t\geq0\right\}$ be a continuous-time Markov chain on a finite state-space $\left\{1,2,\ldots,n\right\}$ with a rate transition matrix $Q=\left[Q_{ij}\right]$ and an initial state probability vector $\eta=\left(\eta_1,\ldots,\eta_n\right)^T$.  In addition,  $\kappa\in\left(0,\infty\right)$ and $u_1,\ldots,u_n>0$ are all constant parameters of the model. In addition, assume that $V(\cdot)$ is given by
	\begin{equation}
	V(s)= \kappa-\int_0^su_{J(t)}dt\ \ , \ \ \forall s\geq0
	\end{equation} 
	and for every $\alpha\in\left[0,\lambda^{-1}\right)$,   
	\begin{equation}
	V_\alpha(s)\equiv V(s)-\frac{\gamma\lambda}{1-\lambda\alpha}s=\kappa-\int_0^su^\alpha_{J(t)}dt\ ,\ \forall s\geq0\,.
	\end{equation}
	such that 
	\begin{equation*}
	u^\alpha_i=u_i+\frac{\gamma\lambda}{1-\lambda\alpha}\ \ , \ \ \forall i=1,2,\ldots,n\,.
	\end{equation*} 
	Clearly, $V_\alpha(\cdot)$ is also a MMFF with the same modulating process $J(\cdot)$ and an initial level $\kappa$ but with other rates $u_1^\alpha,u_2^\alpha,\ldots,u_n^\alpha$ replacing $u_1,u_2,\ldots,u_n$.
	Observe that for a fixed $\alpha\in\left[0,\lambda^{-1}\right)$ and $x\in\left[0,\kappa\right]$
	\begin{equation}
	ES_\alpha(x)=E\inf\left\{s\geq0;\int_0^su_{J(t)}^\alpha dt=\kappa-x\right\}
	\end{equation}
	is the expected time to buffer overflow calculated in Section 6 of \cite{Asmussen2000}. Thus, $x_\alpha$ can be calculated. 
	
	Therefore, it is left to explain how to compute $g(\alpha)$. To this end, notice that 
	\begin{equation}\label{eq: g MMFF}
	g(\alpha)=E\int_0^{S_\alpha}V_\alpha(s)ds=x_\alpha\alpha+E\int_0^{S_\alpha}\hat{V}_\alpha(s)ds
	\end{equation} 
	such that $\hat{V}_\alpha(s)\equiv V_\alpha(s)-x_\alpha,\forall s\geq0$. Therefore, since $\hat{V}_\alpha(\cdot)$ is also a MMFF and
	\begin{equation}
	S_\alpha=\inf\left\{s\geq0;V_\alpha(s)= x_\alpha\right\}=\inf\left\{s\geq0;\hat{V}_\alpha(s)=0\right\}\,,
	\end{equation}
	then $E\int_0^{S_\alpha}\hat{V}_\alpha(s)ds$ is a special case of the expectation derived in Equations (4.7) and (4.8) of \cite{Barron2015} with a discount factor $\beta=0$ (this factor is defined in this reference). Therefore, $g(\alpha)$ can be numerically evaluated. Since Proposition \ref{lemma: existence} implies that it is concave, then it may be maximized by standard techniques. 
	
	\section{Heterogeneous customers with balking}\label{sec: balking}
	This section is about a variation of the model which was introduced in Section \ref{sec: model description} with heterogeneous customers who may balk. 
	
	\subsection{Model description}
	Assume that $(T,\tilde{V}),(T_1,\tilde{V}_1),(T_2,\tilde{V}_2),\ldots$ is an iid sequence of random elements which is independent from the arrival process. For every $i\geq1$, $T_i$ is a random variable which indicates the type of the $i$'th customer. In particular, assume that $T$ has a continuous cdf $F(\cdot)$ such that $\text{Supp}(T)=\left(t_{\text{min}},t_{\max}\right)$ for some $-\infty<t_{\min}<t_{\max}\leq\infty$ on which $F(\cdot)$ is increasing. In addition, $\tilde{V}:\left[t_{\min},t_{\max}\right]\times[0,\infty)\times\Omega\rightarrow[0,\infty)$ is a random field which is:
	\begin{itemize}
		\item Independent from $T$.
		
		\item Nondecreasing continuous in its first coordinate.
		
		\item Nonincreasing right continuous in its second coordinate.
		
		\item  $\tilde{V}(t_{\text{max}},0)$ is a positive square-integrable random variable.
		
		\item $S^{**}\equiv\inf\left\{s\geq0\;\tilde{V}\left(t_{\text{max}},s\right)=0\right\}$ is a square-integrable random variable.
	\end{itemize}  
	Furthermore, for every $i\geq1$, $V_i(\cdot)\equiv \tilde{V}_i(T_i,\cdot)$ is a nonincreasing right-continuous process which describes the marginal utility of the $i$'th customer from service duration. It is also assumed that now the loss rates of the customers from waiting $C_1,C_2,\ldots$ are all equal to a constant $\gamma\in(0,\infty)$.
	
	In this queue,  each customer knows his type when he arrives to the queue. In addition, this type is his private information, \textit{i.e.}, no one else (including the social planner) knows it. The queue is unobservable and he has to decide whether to join or not to join the queue. Importantly, assume that any customer joins the queue if and only if his expected benefit from joining is positive. In addition, if he joins the queue, reneging is impossible and he will observe the evolution of his marginal utility from the initiation of service and until his decision to depart. 
	
	Note that $T_1,T_2,\ldots$ is an iid sequence of random variables which is independent from the arrival process. In addition, $T_i$ is the available information of the $i$'th customer when he makes his join or not to join decision. Therefore, the decisions of the customers whether to join the queue or not are iid and independent from the arrival process. Moreover, by the model assumptions, $V_1(\cdot),V_2(\cdot),\ldots$ is also an iid sequence of random processes which is independent from the arrival process. Thus, since for every $i\geq1$, the service duration of the $i$'th customer is uniquely determined by $V_i(\cdot)$, then the service durations of the customers who join the queue also constitute an iid sequence which is independent from the arrival process. Thus, by known results regarding thinning of a Poisson process, deduce that the resulting system is a $M/G/1$ queue. The purpose is to characterize a price function which is determined uniquely by the service duration and maximizes the long-run average social welfare in the resulting $M/G/1$ system. Note that the maximization is performed over the set of all price functions for which the conditions \textbf{(a)}, \textbf{(b)} and \textbf{(c)} in Section \ref{subsec: social planner} are satisfied.
	
	Importantly, in the original model which was described in Section \ref{sec: model description}, the price function regulates the service demand. However, now the price function has a double role since it regulates simultaneously both the service demand and the arrival rate. More precisely, now the price function implies a selection of customers (by their types) who join the queue and then regulates their service requirements.
	
	\subsection{Optimization}
	Consider an arbitrary price function $p$ which is determined uniquely by the service duration. Assume by contradiction that when $p$ is implemented, customers with type $t\in[t_{\min},t_{\max})$ are joining the queue and customers with type $t'\in(t,t_{\max}]$ are not joining the queue. Note that $F(\cdot)$ is continuous and hence the mass of customers with type $t'$ is zero. Consequently, if all customers with type $t'$ join the queue (and all others don't change their strategies), then their service requirements will not effect the long-run average waiting time.
	Since $\tilde{V}$ is nondecreasing in its first coordinate, this means that customers with type $t'$ are not joining the queue although they gain positive expected benefit from joining which is a contradiction. 
	Therefore, every equilibrium is associated with a threshold $t$ which depends on $p$ such that for every $i\geq1$, the $i$'th customer joins the queue if and only if $T_i>t$.

	Now, for every $t\in[t_{\min},t_{\max})$, denote $P_t\equiv P\left(T>t\right)$ and let $\tau_t\sim T|\{T>t\}$ be a random variable which is independent from $\tilde{V}$. Then, for every $t\in[t_{\min},t_{\max})$ define a process $V_t(\cdot)\equiv\tilde{V}(\tau_t,\cdot)$. The purpose is to solve an optimization 
	
	\begin{equation}\label{optimization: balking}
	v\equiv\max_{t\in[t_{\min},t_{\max})}\max_{S\in\mathcal{D}_{P_t}}\phi_t(S)\,.
	\end{equation}
	where 
	\begin{equation}
	\mathcal{D}_p\equiv\left\{S\in\mathcal{F};S\geq0\,\ ES<\left(\lambda p \right)^{-1}\ ,\ ES^2<\infty\right\}\ \ , \ \ \forall p\in(0,1]
	\end{equation}
	and 
	\begin{equation}
	\phi_t(S)\equiv P_t\left[E\int_0^{S}V_t(s)ds-\frac{\gamma\lambda P_t ES^2}{2\left(1-\lambda P_t ES\right)}\right]\ \ , \ \ \forall t\in[t_{\min},t_{\max})\ ,\ S\in\mathcal{D}_{P_t}\,.
	\end{equation}
	\begin{remark}
		\normalfont Note that $P_{t_{\max}}=0$ and hence the value of the value of the internal optimization in \eqref{optimization: balking} is zero when $t=t_{\max}$. On the other hand, Proposition \ref{lemma: existence} implies that for every $t\in\left[t_{\min},t_{\max}\right)$ the value of the internal optimization in \eqref{optimization: balking} is positive. Hence, in the external optimization, it is enough to consider $t\in\left[t_{\min},t_{\max}\right)$.
	\end{remark}
	To solve \eqref{optimization: balking}, consider two phases: 
	\subsubsection*{Phase A}
	Fix $t\in[t_{\min},t_{\max})$ and consider an optimization
	\begin{equation}
	v(t)\equiv\max_{S\in\mathcal{D}_{P_t}}\phi_t(S)\,.
	\end{equation}
	Then, with a slight abuse of notation, denote its maximizer which follows from  Proposition \ref{lemma: existence} by 
	\begin{equation}
	S_t=\inf\left\{s\geq0;V_{t}(s)-s\frac{\gamma\lambda P_t}{1-\lambda\alpha_t P_t}\leq x_t\right\}
	\end{equation}
	where $\alpha_t=ES_t$ and $x_t\in(0,\infty)$.
	\subsubsection*{Phase B}
	Maximize $v(t)$ on $[t_{\min},t_{\max})$.   
	\begin{theorem}\label{thm: balking}
		There exists $t^*\in\left[t_{\min},t_{\max}\right)$ for which $v=v\left(t^*\right)$.
	\end{theorem}
	\begin{proof}
		See Appendix
	\end{proof}
	
	\begin{remark}
		\normalfont Note that Phase A includes both Phase I and Phase II which are discussed in Section \ref{sec: proof}. 
	\end{remark}
	\subsection{Reverse engineering}
	Let $\pi\in\mathbb{R}$ and define a price function 
	\begin{equation}
	p_{\pi}(s)\equiv \pi+s x_{t^*}+s^2\frac{\gamma\lambda P_{t^*}}{2\left(1- \lambda \alpha_{t^*}P_{t^*}\right)}+\int_0^s\xi(t)dt\ \ , \ \ \forall s\geq0
	\end{equation}
	with a marginal price
	\begin{equation}
	Mp(s)=x_{t^*}+s\frac{\gamma\lambda P_{t^*}}{1-\lambda \alpha_{t^*}P_{t^*}}+\xi(s)\ \ , \ \ \forall s\geq0\,.
	\end{equation}
	Now, consider a strategy profile such that for every $i\geq1$, 
	\begin{enumerate}
		\item The $i$'th customer joins the queue if and only if $T_i>t^*$. 
		
		\item If $T_i>t^*$, then the service duration of the $i$'th customer is given by 
		\begin{equation}
		S_i=\inf\left\{s\geq0;\tilde{V}(T_i,s)\leq Mp(s)\right\}\,.
		\end{equation}
	\end{enumerate}
	From the analysis which was described so far, it turns out that this strategy profile implies an optimal resource allocation from a social point of view. Thus, it is left to find an entry fee $\pi$ which makes the customers join the queue if and only if their types are greater than $t^*$.
	
	Consider a tagged customer with type $t$ and marginal utility $V_i(\cdot)=\tilde{V}(t,\cdot)$ in a stationary system in which $p_\pi(\cdot)$ is implemented and all other customers act according to the above-mentioned strategy profile. It is straightforward that once the tagged customer joins the queue, he will consume service until his marginal utility is less than the marginal price. Thus, it is left to find $\pi$ such that the expected benefit from joining the queue is positive if and only if the type of the tagged customer is greater than $t^*$. The expected benefit of the tagged customer from joining the queue $\Lambda(\pi;t)$ equals to 
	\begin{equation}
	E\int_0^{\inf\left\{s;\tilde{V}(t,s)\leq Mp(s)\right\}}\left[\tilde{V}(t,s)-Mp(s)\right]ds-\frac{\gamma\lambda P(T>t^*) ES_{t^*}^2}{2\left[1-\lambda P\left(T>t^*\right) ES_{t^*}\right]}-\pi\,.
	\end{equation} 
	
	Since $(t,s)\mapsto\tilde{V}(t,s)$ is nondecreasing in its first coordinate, then the expected benefit of the tagged customer from joining is nondecreasing in his type $t$. Therefore, an optimal entry fee $\pi^*$ exists if and only if $\Lambda(\pi;t^*)\leq0$ and $\Lambda(\pi^*;t)>0$ for every $t\in\left(t^*,t_{\max}\right)$. In particular, once $(t,s)\mapsto\tilde{V}(t,s)$ is increasing in its first coordinate, then for every $\pi$, $t\mapsto\Lambda(\pi; t)$ is increasing and hence there exists an optimal entry fee. 
	
	\begin{remark}
		\normalfont When $\pi^*$ exists, the analysis which appears above implies that $p_{\pi^*}(\cdot)$ internalizes the externalities up to an additive constant $\pi^*$. That is, every customer who occupies the server for $s\geq0$ minutes, pays an entry fee $\pi^*$ plus the expected waiting time of other customers which could be saved if he did not join the queue (or joined the queue and reduced his service requirement to zero). 
	\end{remark}
	
	\section{A retrial version of the model} \label{sec: retrial M/G/1}
	Consider the same model which was described in Section \ref{sec: model description} with the following modifications. Assume that now there is no waiting room. Instead, there is an orbit with an infinite capacity such that any customer who finds the server busy at his arrival time joins the orbit. Every customer in the orbit conducts retrials until he finds the server idle and then he starts receiving service. In addition,  the retrial times constitute an iid sequence of exponentially distributed random variables with rate $\theta\in(0,\infty)$ which is independent of all other random elements in this model. In particular, just like in the original model, the service discipline is non-preemptive such that customers are those who decide on their service durations. Assume that $\left(C,D\right),\left(C_1,D_1\right),\left(C_2,D_2\right),\ldots$ is an iid sequence of nonnegative random variables which is independent from all other random elements in this model such that $ED=\delta\in\left(0,\infty\right)$. 
	
	Now, for every $i\geq1$, the total utility of the $i$'th customer from orbiting $w\geq0$ minutes, conducting $r$ retrials and receiving a service of $s\geq0$ minutes is given by 
	\begin{equation}
	U_i \left(s,w,r;p\right)\equiv  X_i(s)-p(s)-C_iw-D_ir\,.
	\end{equation}
	This means that now, besides the original assumptions, there is an additional assumption that the $i$'th customer suffers a constant loss of $D_i$ monetary units per each retrial. 
	
	Given this model, the purpose is to find an optimal price function in same sense of Section \ref{subsec: social planner}. To this end, the same approach which was described in Section \ref{sec: price function} can be carried out. To start with, using known results regarding $M/G/1$ retrial queue with exponential retrial times (see \textit{e.g.}, Equations (3.15) and (3.16) in \cite{Yang1987}) , the optimization to be solved is given by 
	\begin{equation} \label{optimization: social retrial original}
	\max_{S\in\mathcal{D}}:
	\ E\int_0^S V(s)ds-\left( \gamma+\theta\delta\right)\lambda\left[\frac{ES^2}{2\left(1- \lambda ES\right)}+\frac{ES}{\theta\left(1-\lambda ES\right)}\right]\,.
	\end{equation}
	Especially, notice that just like in Section \ref{sec: price function} the objective functional can be extended (denote the extension by $\tilde{f}(\cdot)$) and then maximized on $\mathcal{S}$. This can be done by a two-phase method.
	
	\subsection*{Phase I:}
	For every $\alpha\in\left[0,\lambda^{-1}\right)$ the optimization of Phase I is 
	
	\begin{equation} \label{optimization: social retrial rephrased}
	\begin{aligned}
	& \max_{S\in\mathcal{F}}:
	& &E\int_0^S \left[V(s)-\frac{\lambda\left(\gamma+\theta\delta\right)}{1-\lambda \alpha}s\right]ds \\
	& \ \text{s.t:}
	& & 0\leq S \ , \ P \text{-a.s.}\,, \\ & & & ES=\alpha\,.
	\end{aligned}
	\end{equation}
	This optimization is identical to \eqref{optimization: phase I} up to re-parametrization of $\gamma$ and hence the same results hold for this case. In particular, there exists $x_\alpha$ such that
	\begin{equation}
	T_\alpha\left(x_\alpha\right)\equiv T_\alpha\equiv\inf\left\{s\geq0;V(s)-\frac{\lambda\left(\gamma+\theta\delta\right)}{1-\lambda\alpha}s\leq x_\alpha\right\}
	\end{equation}
	is an optimal solution of \eqref{optimization: social retrial rephrased}. 
	\subsection*{Phase II:}
	For every $\alpha\in\left[0,\lambda^{-1}\right)$ denote the objective value of Phase II:
	
	\begin{align}
	\tilde{g}(\alpha)&\equiv \tilde{f}\left(T_\alpha\right)\\&=E\int_0^{T_\alpha} \left[V(s)-s\frac{\lambda\left(\gamma+\theta\delta\right)}{1-\lambda \alpha}\right]ds-\frac{\lambda\left(\gamma+\theta\delta\right)\alpha}{\theta\left(1-\lambda \alpha\right)}\,.\nonumber
	\end{align}
	This phase can be analyzed by the same method which was applied to Phase II of the original model. 
	
	\begin{theorem}\label{thm: existence retrial}
		There are $\alpha^*\in\left(0,\lambda^{-1}\right)$ and $x^*\in(0,\infty)$ such that $T^*=T_{\alpha^*}\left(x^*\right)$ is a solution of \eqref{optimization: social retrial rephrased}.  In addition, for every constant $\pi$, an optimal price function is given by \begin{equation}\label{eq: pi^* retrial}
		\tilde{p}_\pi^*(s)\equiv \pi+s x^*+s^2\frac{\lambda\left(\gamma+\theta\delta\right)}{2\left(1- \lambda\alpha^*\right)}+\int_0^s\xi(t)dt\ \ , \ \ \forall s\geq0\,.
		\end{equation}
	\end{theorem}
	
	\begin{remark}
		\normalfont Observe that the parameters of $\tilde{p}^*(\cdot)$ can be derived using the numerical procedure which was illustrated in Section \ref{sec: numerical procedure}.
	\end{remark}
	
	\subsection{Externalities in an $M/G/1$ retrial queue}\label{subsec: retrial queue11}
	The purpose of this part is to discuss  externalities in an $M/G/1$ retrial queue with no waiting room, infinite orbit capacity and exponential retrial times with a constant rate (for the exact model setup see, \textit{e.g.}, Sections 2 and 3 of \cite{Yang1987}). A general observation regarding externalities in retrial queues is that unlike the regular $M/G/1$ queue, the externalities caused by a tagged customer are decomposed into two parts:
	
	\begin{enumerate}
		\item \textit{Waiting externalities:} The total waiting time that could be saved for customers if the tagged customer reduced his service demand to zero.
		
		\item \textit{Retrial externalities:} The number of retrials that could be saved for customers if the tagged customer reduced his service demand to zero.
	\end{enumerate}  
	Consider the case when $V(\cdot)$ is a continuous process. Then,  by the same kind of analysis which was performed in the proof of Lemma \ref{lemma: step1}, it can be deduced that
	\begin{align}\label{eq: x^* retrial}
	x^*=\frac{\lambda\left(\gamma+\theta\delta\right)}{1-\lambda\alpha^*}\left[\frac{\lambda\left(\frac{E(T^*)^2}{2}+\frac{\alpha^*}{\theta}\right)}{1-\lambda\alpha^*}+\theta^{-1}\right]\,.
	\end{align}
	To see this, assume that $V(\cdot)$ is a continuous process. Then, observe that
	\begin{equation}
	V(T^*)=x^*+\frac{\lambda(\gamma+\theta\delta)}{1-\lambda\alpha^*}T^*
	\end{equation}
	and hence
	\begin{equation}\label{eq:lastequation}
	ET^*V(T^*)=x^*\alpha^*+\frac{\lambda(\gamma+\theta\delta)}{1-\lambda\alpha^*}E(T^*)^2\,.
	\end{equation}
	In addition, for every $v\in\mathbb{R}$ define 
	
	\begin{align}
	\tilde{\psi}(v)&\equiv \tilde{f}\left(T^*v\right)\\&=E\int_0^{T^*v} \left[V(s)-s\frac{\lambda\left(\gamma+\theta\delta\right)}{1-v\lambda \alpha^*}\right]ds-\frac{\lambda\left(\gamma+\theta\delta\right)v\alpha^*}{\theta\left(1-v\lambda \alpha^*\right)}\,.\nonumber
	\end{align}
	It is known that $\psi(\cdot)$ attains a maximum at $v=1$. In addition, since $V(\cdot)$ is continuous, it is differentiable on a neighborhood of $v=1$. Thus, the corresponding first order condition is given by 
	\begin{align}
	0&=\tilde{\psi}'(v=1)=ET^*V(T^*)\\&\nonumber-\lambda(\gamma+\theta\delta)\left[\frac{E(T^*)^2}{1-\lambda\alpha^*}+\frac{\lambda\alpha^*E(T^*)^2}{2(1-\lambda\alpha^*)^2}+\frac{\alpha^*}{\theta(1-\lambda\alpha^*)}+\frac{\lambda(\alpha^*)^2}{\theta(1-\lambda\alpha^*)^ 2}\right]\,.
	\end{align}
	By reordering this equation and using \eqref{eq:lastequation} deduce that
	\begin{equation}
	0=\alpha^*\left\{x^*-\frac{\lambda(\gamma+\theta\delta)}{1-\lambda\alpha^*}\left[\frac{\lambda\left(\frac{E(T^*)^2}{2}+\frac{\alpha^*}{\theta}\right)}{1-\lambda\alpha^*}+\theta^{-1}\right]\right\}\,.
	\end{equation}
	Thus, $\alpha^*>0$ implies \eqref{eq: x^* retrial}. Now, by an insertion of \eqref{eq: x^* retrial} into \eqref{eq: pi^* retrial} deduce that for every $s\geq0$
	\begin{align}
	\tilde{p}^*_0(s)=\lambda\left(\gamma+\theta\delta\right)z(s)+\int_0^s\xi(t)dt
	\end{align}
	such that 
	\begin{equation}
	z(s)=\frac{s}{1-\lambda\alpha^*}\left[\frac{\lambda\left(\frac{E(T^*)^2}{2}+\frac{\alpha^*}{\theta}\right)}{1-\lambda\alpha^*}+\theta^{-1}\right]+s^2\frac{1}{2\left(1-\lambda\alpha^*\right)}\,.
	\end{equation}
	\begin{remark}
		\normalfont In fact, it is possible to use the same technique which was applied in the proof of Lemma \ref{lemma: x identity}  in order to generalize this result to the case where $V(\cdot)$ might have jumps downwards. 
	\end{remark}
	
	\subsubsection*{Conjecture regarding expected externalities}
	Because this model is similar to the original one, it makes sense that $\tilde{p}^*_0(\cdot)$ internalizes the externalities in this model just like $p^*_0(\cdot)$ does in the original one. Thus, due to the interpretations of $\gamma$ and $\delta$, it is plausible that the expected waiting and retrial  externalities which are caused due to a tagged customer with a service demand of $s\geq0$ minutes are given respectively by $\lambda z(s)$ and $\lambda\theta z(s)$. Since this is not the focus of this work, the proof of this conjecture is left for future research.
	
	\section{Future research}\label{sec: future research}
	This work is focused on  the problem of finding an optimal regulation scheme for queueing models with customers who dynamically decide on their service durations. To the best of the author' s knowledge, this problem has never been discussed in the strategic-queueing literature (for a recent survey, see, \textit{e.g.,} \cite{Hassin2016}). Hence, while this work includes some results, there is still a big room for future research. For example, consider the following questions:
	\begin{enumerate}
		
		\item Are there cases in which the marginal utility is not nonincreasing and the optimal price function can be derived? 
		
		\item What are the results in a similar analysis with several servers or in a processor-sharing system? 
		
		\item Is it possible to derive results  for the case where the arrival process is not Poisson?
		
		\item In the context of Section \ref{sec: balking}, what can be said regarding a similar model with observable queue?  
		
		\item Is it possible to derive analogue results for a model in which reneging is allowed? 
	\end{enumerate}

	\subsubsection*{Acknowledgement:} The author is extremely grateful to Binyamin Oz for valuable discussions as well as for sharing his unpublished manuscript. In addition, the
	author would like to thank Offer Kella, Moshe Haviv and Refael Hassin for
	their comments before the submission. 
	
\section*{Appendix}
\subsection*{Proof of Proposition \ref{prop: long run}}
Assume that $ES_1>0$. For every $k\geq1$ let $N_k$ be the number of customers who receive service during the $k$'th busy period. 
$C_1,C_2,\ldots$ are independent of all other random quantities in this model. Therefore, by conditioning and un-conditioning with respect to this sequence, then a known result regarding the long-run average queue-length in an $M/G/1$ queue implies that the long-run average loss due to waiting time equals to 

\begin{equation}\label{eq: long run 1}
\frac{\gamma\lambda^2 ES_1^2}{2\left(1-\lambda ES_1\right)}\,.
\end{equation}
In addition, $V(\cdot)$ is a nonincreasing process such that $V(0)>0$. Hence, the details $ES_1^2<\infty$ and $EV^2(0)<\infty$ imply that
\begin{align}
E\left[\int_0^{S_1}V_1(s)ds\right]^+ds\leq E\int_0^{S_1}V_1^+(s)ds\leq ES_1V_1(0)<\infty\,.\nonumber
\end{align}
Thus, $E\int_0^{S_1}V_1(s)ds$ is well-defined. Moreover, for every $i\geq1$, $S_i$ is determined by $V_i(\cdot)$ as a solution of \eqref{optimization: individual}. Thus, since $V_1(\cdot),V_2(\cdot),V_3(\cdot)$ is an iid sequence, then  $\left(S_1,V_1(\cdot)\right),\left(S_2,V_2(\cdot)\right),\left(S_3,V_3(\cdot)\right),\ldots$ is an iid sequence. Hence 
\begin{equation}
\int_0^{S_1}V_1(s)ds,\int_0^{S_2}V_2(s)ds,\int_0^{S_3}V_3(s)ds,\ldots
\end{equation}
is also an iid sequence and the rest follows by some standard renewal-reward arguments.  $\blacksquare$

\subsection*{Auxiliary lemmata}
The following auxiliary  Lemma \ref{lemma: approximation} and Lemma \ref{lemma: leibnitz} will be used later on in the proof of Proposition \ref{lemma: existence}. 
\begin{lemma} \label{lemma: approximation}
Let $v:\left[0,\infty\right)\rightarrow\left[0,\infty\right)$ be a nonincreasing right-continuous function. In addition, for every $n\geq1 $ define $v_n(s)\equiv v\ast g_n(s),\forall s\geq0$ such that
\begin{equation}
g_n(u)\equiv n1_{\left[0,\frac{1}{n}\right]}(-u)\ \ ,\ \ \forall u\in\mathbb{R}\,.
\end{equation} 
Then, 
\begin{enumerate}
	\item For every $n\geq1 $, $v_n(\cdot)$ is a continuous, nonnegative and nonincreasing function on $\left[0,\infty\right)$ such that $0\leq v_n(0)\leq v(0)$.
	
	\item For every $s\in\left[0,\infty\right)$,  $v_n(s)\uparrow v(s)$ as $n\to\infty$.   
\end{enumerate} 
\end{lemma}

\begin{proof}
\begin{enumerate}
	\item Let $U$ be a random variable which is distributed uniformly on $\left[0,1\right]$ and for every $n\geq1 $ denote $U_n\equiv \frac{U}{n}$. Fix $n\geq1 $ and notice that 
	\begin{equation}
	v_n(s)=Ev\left(s+U_n\right)\ \ , \ \ \forall s\geq0\,.
	\end{equation}
	Recall that $v(\cdot)$ is nonnegative and nonincreasing function, \textit{i.e.,} $v_n(\cdot)$ also shares these properties. In addition, since $U_n$ is nonnegative, then $v_n(0)\leq v(0)$. To show that $v_n(\cdot)$ is continuous on $\left[0,\infty\right)$, pick an arbitrary $s\in\left[0,\infty\right)$ and let $(s_k)_{k=1}^\infty$ be a sequence such that $s_k\to s$ as $k\to\infty$. Then, since $v(\cdot)$ is nonincreasing, $s-U_n$ is $P$-a.s. a continuity point of $v(\cdot)$, \textit{i.e.},  $v\left(s_k+U_n\right)\to v\left(s+U_n\right)$ as $k\to\infty$ when the convergence holds $P$-a.s. In addition, $0\leq v\left(s_k+U_n\right)\leq v(0)<\infty$ for every $k\geq1$ and hence dominated convergence theorem implies that 
	\begin{equation}
	\lim_{k\to\infty}v_n(s_k)=E\lim_{k\to\infty} v\left(s_k+U_n\right)=Ev\left(s+U_n\right)=v_n(s)
	\end{equation}   
	and the result follows. 
	
	\item Fix $s\in\left[0,\infty\right)$ and observe that  $v(\cdot)$ is nonnegative, nonincreasing and right-continuous. Thus, since $U_n\downarrow0$ as $n\to\infty$, then the result follows by monotone convergence theorem.    
\end{enumerate}
\end{proof}

\begin{lemma}\label{lemma: leibnitz}
Let $\left(\textbf{X},\mathcal{X},\mu\right)$ be a general measure space and let $\alpha:\textbf{X}\rightarrow\mathbb{R}$, $\xi:\textbf{X}\times[0,\infty)\rightarrow\mathbb{R}$ such that
\begin{description}
	
	\item[(i)] For every $t\geq0$, $x\mapsto\xi(x,t)$ is $\mathcal{X}$-measurable.
	
	\item[(ii)] For every $x\in\mathbf{X}$, $t\mapsto\xi(x,t)$ is right-continuous on $[0,\infty)$.  
\end{description}
In addition, let $\beta(\cdot)\in L_1(\mu)$ and define 
\begin{equation*}
\zeta(x,t)=\beta(x)+\int_0^t\xi(x,s)d s\ \ , \ \ \forall (x,t)\in\mathbf{X}\times[0,\infty)\,,
\end{equation*} 
\begin{equation*}
\varphi(t):=\int_{\mathbf{X}}\zeta(x,t)\mu(d x) \ \ , \ \ \forall t\in[0,\infty)\,.
\end{equation*}
If at least one of the following conditions hold:
\begin{description}
	\item[C1:] For every $x\in\mathbf{X}$, $t\mapsto\xi(x,t)$ is nonnegative and nonincreasing.
	
	\item[C2:] There exists $\psi(\cdot)\in L_1(\mu)$ such that $|\xi(x,t)|\leq |\psi(x)|$ for every $(x,t)\in \textbf{X}\times[0,\infty)$. 
\end{description}
Then, $\varphi(\cdot)$ is right-differentiable on $[0,\infty)$ such that 
\begin{equation}\label{eq: right-derivative}
\partial_+\varphi(t)=\int_{\textbf{X}}\xi(x,t)\mu(d x)\ \ , \ \ \forall t\in[0,\infty)\,.
\end{equation} 
Moreover, if \textbf{C2} holds and $t\mapsto\xi(x,t)$ is continuous on $(0,\infty)$, then $\varphi(\cdot)$ is differentiable on $(0,\infty)$ such that
\begin{equation}\label{eq: derivative}
\frac{d}{d t}\varphi(t)=\int_{\textbf{X}}\xi(x,t)\mu(d x)\ \ , \ \ \forall t\in(0,\infty)\,.
\end{equation} 
\end{lemma}

\begin{proof}
For simplicity and without loss of generality the proof is given for the case where $\beta$ is identically zero. In addition, let $\mathcal{B}[0,\infty)$ be a notation of the Bor\'el $\sigma$-field which is associated with $[0,\infty)$. Notice that assumptions \textbf{(i)} and \textbf{(ii)} imply that $(x,t)\mapsto\xi(x,t)$ is $\mathcal{X}\otimes\mathcal{B}[0,\infty)$-measurable. For details see, \textit{e.g.}, Remark 1.4  on page 5 of \cite{Karatzas1988}.

Now,  observe that under either \textbf{C1} or \textbf{C2}, Fubini's theorem may be applied in order to deduce that
\begin{equation*}
\eta(t):=\int_{\mathbf{X}}\xi(x,t)\mu(d x)\ \ , \ \ \forall t\in[0,\infty)
\end{equation*}
is $\mathcal{B}[0,\infty)$-measurable and satisfies 
\begin{align*}
\varphi(t)&=\int_{\mathbf{X}}\zeta(x,t)\mu(d x)\\&=\int_{\mathbf{X}}\int_0^t\xi(x,s)d s\mu(d x)=\int_0^t\eta(s)d s\ \ , \ \ \forall t\in[0,\infty)\,.
\end{align*}
Thus, in order to show \eqref{eq: right-derivative}, it is enough to prove that $\eta(\cdot)$ is right-continuous on $[0,\infty)$. To this end, by \textbf{(ii)}, under \textbf{C1} (\textbf{C2}) monotone (dominated) convergence theorem implies that
\begin{equation*}
\lim_{s\downarrow t}\eta(s)=\int_\textbf{X}\lim_{s\downarrow t}\xi(x,s)d s=\int_\textbf{X}\xi(x,t)d s=\eta(t)\ \ , \ \ \forall t\in[0,\infty)
\end{equation*}
and the result follows. Finally, observe that in order to show \eqref{eq: derivative}, it is enough to prove that $\eta(\cdot)$ is continuous on $(0,\infty)$. This can be done by using similar arguments.	
\end{proof}	

\begin{remark}
	\normalfont Note that once $\left(\textbf{X},\mathcal{X},\mu\right)$ is complete, the conclusion of Lemma \ref{lemma: leibnitz} remains valid even when the requirements which appear in \textbf{(ii)}, \textbf{C1} and \textbf{C2} are satisfied $\mu$-a.s. (instead of pointwise). 
\end{remark}

\subsection*{Proof of Proposition \ref{lemma: existence}} 
The proof of Proposition \ref{lemma: existence} is given by the subsequent lemmata:
\begin{lemma}\label{lemma: nonnegative x}
	Assume that $\alpha\in\left[0,\lambda^{-1}\right)$ such that $x_\alpha<0$. Then, there exists $\tilde{\alpha}\in\left[0,\lambda^{-1}\right)$ such that $x_{\tilde{\alpha}}\geq0$ and $g\left(\alpha\right)\leq g\left(\tilde{\alpha}\right)$. 
\end{lemma}
\begin{proof}
	Let $\alpha$ be such that $x_\alpha<0$ and define 
	\begin{equation}
	\tilde{S}\equiv\inf\left\{s\geq0;V(s)-\frac{\gamma\lambda}{1-\lambda\alpha}s\leq0\right\}\,.
	\end{equation}
	Notice that $\tilde{S}\leq S_\alpha$ and hence $\tilde{\alpha}\equiv E\tilde{S}\leq \alpha$. This implies that
	\begin{align}
	g\left(\tilde{\alpha}\right)&\geq f\left(\tilde{S}\right)\\&=E\int_0^{\tilde{S}}\left[V(s)-\frac{\gamma\lambda}{1-\lambda\tilde{\alpha}}s\right]ds\nonumber\\&\geq E\int_0^{\tilde{S}}\left[V(s)-\frac{\gamma\lambda}{1-\lambda\alpha}s\right]ds\nonumber\\&\geq E\int_0^{S_{\alpha}}\left[V(s)-\frac{\gamma\lambda}{1-\lambda\alpha}s\right]ds=g\left(\alpha\right)\,.\nonumber
	\end{align}
	In addition, notice that $\tilde{\alpha}\leq\alpha$, $\tilde{\alpha}=E\tilde{S}=ES_{\tilde{\alpha}}$ and  
	\begin{equation}
	S_{\tilde{\alpha}}=\inf\left\{s\geq0;V(s)-\frac{\gamma\lambda}{1-\lambda\tilde{\alpha}}s\leq x_{\tilde{\alpha}}\right\}\,.
	\end{equation}
	Thus, since $\alpha\mapsto x_\alpha$ is nonincreasing, then  $x_{\tilde{\alpha}}$ is nonnegative and the result follows.
\end{proof}

\begin{lemma}\label{lemma: S bound}
	Consider some $\alpha\in\left[0,\lambda^{-1}\right)$.
	\begin{enumerate}
		\item If $x_\alpha\geq0$,  then, $0\leq S_\alpha\leq\frac{V(0)}{\gamma\lambda}$, and $g(\alpha)\geq0$.
		
		\item For every $\alpha\in\left[0,\lambda^{-1}\right)$
		\begin{equation}
		g\left(\alpha\right)\leq\frac{EV^2(0)}{\gamma\lambda}<\infty\,.
		\end{equation}
	\end{enumerate}  
\end{lemma} 

\begin{proof}
	Let $\alpha\in\left[0,\lambda^{-1}\right)$ such that $x_\alpha\geq0$ and notice that
	\begin{align}
	S_\alpha&=\inf\left\{s\geq0;V(s)-\frac{\gamma\lambda}{1-\lambda\alpha}s\leq x_\alpha\right\}\\&\leq \inf\left\{s\geq0;V(s)-\gamma\lambda s\leq 0\right\}\,.
	\end{align}
	Therefore, since $V(\cdot)$ is nonincreasing, deduce that 
	\begin{equation}
	0\leq S_\alpha\leq \frac{V(0)}{\gamma\lambda}
	\end{equation}  
	and hence
	\begin{align}
	g\left(\alpha\right)&=E\int_0^{S_\alpha}\left[V(s)-\frac{\gamma\lambda}{1-\lambda\alpha}s\right]ds\\&\leq EV(0)S_\alpha\leq\frac{EV^2(0)}{\gamma\lambda}<\infty\nonumber\,.
	\end{align}
	Then, use Lemma \ref{lemma: nonnegative x} in order to show that this upper bound holds for every $\alpha\in\left[0,\lambda^{-1}\right)$. Finally, to show that $g(\alpha)\geq0$ for every $\alpha$ for which $x_\alpha\geq0$ observe that in such a case, $g(\alpha)$ is defined as an expectation of an integral with an integrand which is nonnegative on the integration domain.   
\end{proof}
\begin{lemma}\label{lemma: square-integrable}
	For every $\alpha\in\left[0,\lambda^{-1}\right)$, $S_\alpha$ is square-integrable.
\end{lemma}
\begin{proof}
	Let $\alpha\in\left[0,\lambda^{-1}\right)$. If $x_\alpha\geq0$, then the result is a consequence of Lemma \ref{lemma: S bound} and hence it is left to consider the case when $x_\alpha\in(-\infty,0)$. To this end,  define $\tilde{V}(s)=V(s)-x_\alpha,\forall s\geq0$ and observe that for every $S\in\mathcal{S}$ such that $ES=\alpha$, 
	\begin{equation*}
	E\int_0^S\left[\tilde{V}(s)-\frac{\gamma\lambda}{1-\lambda\alpha}s\right]ds=-\alpha x_\alpha+f\left(S\right)\,.
	\end{equation*}
	This means that $S_\alpha$ is also a solution of
	\begin{equation} 
	\begin{aligned}
	& \max_{S\in\mathcal{F}}:
	& &E\int_0^S\left[ \tilde{V}(s)-\frac{ \gamma\lambda}{1- \lambda\alpha}s\right]ds \\
	& \ \text{s.t:}
	& & 0\leq S \ , \ P \text{-a.s.}\,, \\ & & & ES=\alpha\,.
	\end{aligned}
	\end{equation}
	In addition, note that $\tilde{V}(\cdot)$ is nonincreasing right-continuous process such that $\tilde{V}(0)=V(0)-x_\alpha$ is a square-integrable positive random variable. Consequently, Lemma \ref{lemma: S bound} implies the result because by definition
	\begin{equation}
	S_\alpha=\inf\left\{s\geq0;\tilde{V}(s)-\frac{\gamma\lambda}{1-\lambda\alpha}s\leq0\right\}\,.
	\end{equation}	  
\end{proof}

Note that for every $S\geq0$ which is square integrable
\begin{equation}\label{eq: f square-integrable}
f(S)=E\int_0^SV(s)ds-\frac{\gamma\lambda ES^2}{2\left(1-\lambda\alpha\right)}\,.
\end{equation}

\begin{lemma} \label{lemma: phase II convex}
	$f(\cdot)$ is concave on $\mathcal{D}$ and
	$g(\cdot)$ is concave on $\left[0,\lambda^{-1}\right)$.	
\end{lemma}

\begin{proof}
	Define 	
	\begin{equation}
	\mathcal{S}_0\equiv\mathcal{D}\times\left[0,\lambda^{-1}\right)
	\end{equation} 
	and for every $(S,\alpha)\in\mathcal{S}_0$ denote 
	\begin{equation}
	h(S,\alpha)\equiv  \int_0^SV(s)ds-\frac{\gamma\lambda S^2}{2\left(1-\lambda\alpha\right)}\,.
	\end{equation}
	In particular, observe that $V(\cdot)$ is nonincreasing right-continuous process which implies that $s\mapsto\int_0^sV(t)dt$ is concave on $[0,\infty)$. Therefore, since $(t,s)\mapsto\frac{t^2}{s}$ is convex on $\mathbb{R}\times(0,\infty)$, then $h(S,\alpha)$ is concave on $\mathcal{S}_0$. Thus, since an expectation is a linear operator, then  
	\begin{equation}
	H(S,\alpha)\equiv Eh(S,\alpha)=E\int_0^SV(s)ds-\frac{\gamma\lambda ES^2}{2\left(1-\lambda\alpha\right)}\ , \ \forall(S,\alpha)\in\mathcal{S}_0
	\end{equation} 
	is a concave functional on $\mathcal{S}_0$. Especially, notice that $(S,\alpha)\in\mathcal{S}_0$ implies that $S$ is square-integrable and hence \eqref{eq: f square-integrable} could be used in order to justify the last equality. Now, consider $S_1,S_2\in\mathcal{D}$ and for every $i=1,2$ denote $\alpha_i\equiv ES_i$. Thus, observe that the concavity of $H(\cdot)$ implies that for every $\mu\in(0,1)$
	\begin{align*}
	f\left[S_2+\mu(S_1-S_2)\right]&=H\left[S_2+\mu(S_1-S_2),\alpha_2+\mu(\alpha_1-\alpha_2)\right]\\&\geq\mu H\left(S_1,\alpha_1\right)+(1-\mu)H\left(S_2,\alpha_2\right)\\&=\mu f(S_1)+(1-\mu)f(S_2)
	\end{align*}
	and hence the concavity of $f(\cdot)$ follows by definition. 
	
	In order to prove the concavity of $g(\cdot)$, recall Lemma \ref{lemma: square-integrable} which implies that for every $\alpha\in[0,\lambda^{-1})$
	\begin{equation} \label{eq: f(alpha) definition}
	g(\alpha)=\sup\bigg\{H(S,\alpha);S\in\mathcal{D}\ , \ ES=\alpha\bigg\}
	\end{equation}  
	and hence the result follows because $g(\cdot)$ equals to a supremum of a concave functional on a convex set which is not empty (take, \textit{e.g.}, $(\alpha,\alpha)$).
\end{proof}

\begin{lemma}\label{lemma: phase II continuous}
	$\lim_{\alpha\downarrow0}g(\alpha)=g(0)=0$.  
\end{lemma}
\begin{proof}
	Denote
	\begin{equation}
	S^0\equiv\inf\left\{s\geq0;V(s)\leq0\right\}
	\end{equation}
	and $\alpha_0\equiv E\left(S^0\wedge\frac{1}{2\lambda}\right)$. Observe that positiveness of $S^0$ implies that $\alpha_0\in\left(0,\frac{1}{2\lambda}\right]$. Then, for every $\alpha\in[0,\alpha_0)$, define $\hat{S}_\alpha\equiv\frac{\alpha}{\alpha_0}\left(S^0\wedge\frac{1}{2\lambda}\right)$ which is square-integrable nonnegative random variable such that $E\hat{S}_\alpha=\alpha$. Thus, by the definition of $g(\cdot)$ and using \eqref{eq: f square-integrable} deduce that  
	\begin{equation}\label{eq: f(alpha) lower bound}
	g(\alpha)\geq E\int_0^{\hat{S}_\alpha} V(s)ds-\frac{ \alpha^2 \gamma\lambda }{2\alpha_0^2\left(1- \lambda\alpha\right)}E\left(S^0\wedge\frac{1}{2\lambda}\right)^2\,.
	\end{equation}
	In particular, the expectation in the second term is finite and hence this term tends to zero as $\alpha\downarrow0$. In addition, for every $\alpha\in\left[0,\alpha_0\right)$,
	
	\begin{align*}
	0\leq \int_0^{\hat{S}_\alpha} V(s)ds\leq \int_0^{S^0\wedge\frac{1}{2\lambda}}V(0)ds\leq \frac{V(0)}{2\lambda}
	\end{align*}
	Thus, since $EV(0)<\infty$, dominated convergence implies that the first term in \eqref{eq: f(alpha) lower bound} tends to zero as $\alpha\downarrow0$. To provide an upper bound which tends to zero, note that for every $\alpha\in\left[0,\lambda^{-1}\right)$
	\begin{align}
	g(\alpha)&=E\int_0^{S_\alpha}\left[V(s)-s\frac{\gamma\lambda}{1-\lambda \alpha}\right]ds\\&\leq E\int_0^{S_\alpha}\left[V(s)-s\gamma\lambda\right]ds\nonumber\\&\leq E\int_0^{\tilde{S}_\alpha}\left[V(s)-s\gamma\lambda\right]ds\nonumber
	\end{align}
	where $\tilde{S}_\alpha$ is the solution of 
	\begin{equation}
	\begin{aligned}
	& \max_{S\in\mathcal{F}}:
	& &E\int_0^S\left[ V(s)-s\gamma\lambda\right]ds \\
	& \ \text{s.t:}
	& & 0\leq S \ , \ P \text{-a.s.}\,, \\ & & & ES=\alpha
	\end{aligned}
	\end{equation}
	which is specified by Theorem 1 of \cite{Jacobovic2020}. In particular, notice that this optimization is well-defined due to \eqref{eq: regularity condition 0}. In addition, note that existence of this solution is justified by the same kind of argument which was provided in order to justify that $S_\alpha$ is a solution of \eqref{optimization: phase I}. Now, let 
	\begin{equation}\label{eq: S tilde}
	\tilde{S}\equiv\inf\left\{s\geq0;\ V(s)-s\gamma\lambda\leq0\right\}
	\end{equation}
	and notice that  
	\begin{equation}
	E\int_0^{\tilde{S}}\left[\ V(s)-s\gamma\lambda\right]^-ds=0<\infty\,.
	\end{equation}
	Therefore, the pre-conditions of Proposition 1 of \cite{Jacobovic2020} are satisfied, \textit{i.e.}, 
	\begin{equation}
	\exists\lim_{\alpha\downarrow0}E\int_0^{\tilde{S}_\alpha}\left[V(s)-s\gamma\lambda\right]ds=0
	\end{equation}
	and the proof is completed.
\end{proof}
\begin{lemma}\label{lemma: alpha bound}
	Let $\alpha'\equiv\inf\{\alpha\in[0,\lambda^{-1});x_\alpha<0\}$. Then, $\alpha'<\lambda^{-1}$.
\end{lemma}
\begin{proof}
	Assume by contradiction that $\alpha'=\lambda^{-1}$ which means that $x_\alpha\geq0$ for every $\alpha\in\left[0,\lambda^{-1}\right)$. In addition, observe that Lemma \ref{lemma: square-integrable} and \eqref{eq: f square-integrable} imply that for every $\alpha\in\left[0,\lambda^ {-1}\right)$
	\begin{equation*}
	g\left(\alpha\right)=E\int_0^{S_{\alpha}}V(s)ds-\frac{\gamma ES_{\alpha}^2}{2\left(1-\lambda\alpha\right)}\,.
	\end{equation*} 
	In addition,   $x_{\alpha}\geq0,\forall \alpha\in\left[0,\lambda^{-1}\right)$ and hence Lemma \ref{lemma: S bound} implies that $0\leq S_{\alpha}\leq\frac{V(0)}{\gamma\lambda},\forall \alpha\in\left(\alpha_0,\lambda^{-1}\right)$. Therefore, since $V(\cdot)$ is nonincreasing, deduce that for every $\alpha\in\left[0,\lambda^{-1}\right)$
	\begin{equation}
	E\int_0^{S_{\alpha}}V(s)ds\leq EV(0)S_{\alpha}\leq\frac{EV^2(0)}{\gamma\lambda}\,.
	\end{equation}
	These results imply that for every $\alpha\in\left[0,\lambda^{-1}\right)$
	
	\begin{equation}\label{eq: upper bound}
	g\left(\alpha\right)\leq\frac{EV^2(0)}{\gamma\lambda}-\frac{\gamma ES^2_{\alpha}}{2\left(1-\lambda\alpha\right)}\,.
	\end{equation}
	Now, if 
	\begin{equation}
	\lim_{\alpha\uparrow\lambda^{-1}}\inf ES_\alpha^2>0\,,
	\end{equation}
	then \eqref{eq: upper bound} implies that $g(\alpha)$ tends to $-\infty$ as $\alpha\uparrow\lambda^{-1}$. Thus, in such a case there exists $\alpha\in(0,\lambda^{-1})$ such that $g(\alpha)<0$. On the other hand, recall that $x_{\alpha}\geq0$ and 
	\begin{equation}
	S_{\alpha}=\inf\left\{s\geq0;V(s)-\frac{\lambda\gamma s}{1-\lambda\alpha}\leq x_{\alpha}\right\}\,.
	\end{equation}
	Thus, since $V(\cdot)$ is nonincreasing right-continuous process, then this implies that
	\begin{equation}
	g(\alpha)=f(S_{\alpha})=E\int_0^{S_{\alpha}}\left[V(s)-\frac{\gamma\lambda s}{1-\lambda\alpha}\right]ds\geq0
	\end{equation}
	which implies a contradiction.
	
	Hence, deduce that
	\begin{equation}
	\lim_{\alpha\uparrow\lambda^{-1}}\inf ES_\alpha^2=0\,.
	\end{equation}
	Then, since for every $\alpha\in[0,\lambda^{-1})$, $S_\alpha$ is squared-integrable, then the Cauchy-Schwartz inequality leads to the following contradiction
	\begin{align}
	0<\lambda^{-2}&=\lim_{\alpha\uparrow\lambda^{-1}}\inf\alpha^2\\&=\lim_{\alpha\uparrow\lambda^{-1}}\inf \left(ES_\alpha\right)^2\nonumber\\&\leq \lim_{\alpha\uparrow\lambda^{-1}}\inf ES_\alpha^2=0\,.
	\end{align}
	 
\end{proof}
\begin{lemma}\label{lemma: existence1}
	There exists $\alpha^*\in\left[0,\frac{EV(0)}{\gamma\lambda}\right]\cap\left[0,\lambda^{-1}\right)$ which is a maximizer of $g(\cdot)$ on $[0,\lambda^{-1})$ such that $x_{\alpha^*}\geq0$ and $S^*\equiv S_{\alpha^*}$ is an optimal solution of \eqref{optimization: social1} which is square-integrable.
\end{lemma}
\begin{proof}
	By Lemma \ref{lemma: alpha bound}, $\alpha'<\lambda^{-1}$ and hence Lemma \ref{lemma: nonnegative x} implies that Phase II is reduced to maximization of $g(\cdot)$ on the closed interval $\left[0,\alpha'\right]$. By Lemma \ref{lemma: phase II convex} and Lemma \ref{lemma: phase II continuous} deduce that $g(\cdot)$ is continuous on $[0,\alpha']$ and hence there exists $\alpha^*\in[0,\alpha']$ which maximizes the value of $g(\cdot)$ over $[0,\lambda^{-1})$. In addition, given the maximizer $\alpha^*$, square integrability of $S^*$ is a direct consequence of Lemma \ref{lemma: square-integrable}. Finally, the upper bound and the fact that $x_{\alpha^*}\geq0$ stems immediately from Lemma \ref{lemma: nonnegative x} and Lemma \ref{lemma: S bound}.	
\end{proof}

\begin{lemma}\label{lemma: positiveness}
	$f(S^*)=0$ if and only if $\alpha^*=0$.
\end{lemma}
\begin{proof}
	Assume that $0=\alpha^*=ES^*$. Since $S^*$ is a nonnegative random variable, then $S^*=0$, $P$-a.s. and  $f(S^*)=0$ follows immediately. To show the other direction assume that 
	\begin{equation}
	0=f(S^*)=E\int_0^{S^*}\left[V(s)-\frac{\gamma\lambda s}{1-\lambda\alpha^*}\right]ds
	\end{equation}
	and recall that
	\begin{equation}
	S^*=\inf\left\{s\geq0;V(s)-\frac{\gamma\lambda s}{1-\lambda\alpha^*}\leq x^*\right\}\,.
	\end{equation}
	Therefore, since $V(\cdot)$ is nonincreasing right-continuous process, then
	\begin{equation*}
	0\geq E\int_0^{S^*}x^*=\alpha^*x^*\,.
	\end{equation*}
	and hence $x^*\geq0$ implies that $\alpha^*=0$ (remember that $\alpha^*\in[0,\lambda^{-1})$).
\end{proof}
\begin{equation*}
\end{equation*}

Observe that Lemma \ref{lemma: existence1} implies that there exists $\alpha^*\in\left[0,\lambda^{-1}\right)$ and $x^*= x_{\alpha^*}\geq0$ for which $S^*=S_{\alpha^*}=S_{\alpha^*}\left(x^*\right)$ is an optimum of \eqref{optimization: social1}. Therefore, since $V(\cdot)$ is nonincreasing and right-continuous, then its left limit at $S^*$ is nonnegative and hence $S^*$ is also an optimum (with the same objective value) of the analogue optimization with $V^+(\cdot)$ replacing $V(\cdot)$. Thus, without loss of generality, from now on assume that $V(\cdot)$ is nonnegative. 

\begin{lemma}
	$f(S^*)>0$ (and hence $\alpha^*>0$).
\end{lemma}

\begin{proof}
	In order to prove that $f(S^*)>0$ it is enough to find a random variable $S_0\in\mathcal{S}$ for which $f(S_0)>0$. To this end, for every $\alpha\in[0,\infty)$ define a function 
	\begin{equation}
	\upsilon(\alpha)\equiv f(\alpha)=E\int_0^\alpha V(s)ds-\frac{\gamma\lambda \alpha^2}{2(1-\lambda\alpha)}\,.
	\end{equation}
	Since $V(\cdot)$ is nonnegative nonincreasing  right-continuous process, then Lemma \ref{lemma: leibnitz} implies that $\upsilon(\cdot)$ is right-differentiable with a right-derivative at zero which equals to
	\begin{equation}
	\partial_+\upsilon(\alpha=0)=EV(0)>0\,.
	\end{equation}
	This means that there exists $\alpha_0>0$ such that $\upsilon(\alpha_0)>\upsilon(0)=0$ and the result follows.  
\end{proof}

\begin{lemma}\label{lemma: step1}
	If $V(\cdot)$ is continuous, then \eqref{eq: x equation} holds.
\end{lemma}
\begin{proof}
	Assume that $V(\cdot)$ is a continuous process and denote $x^*=x_{\alpha^*}$. Thus, since $V(\cdot)$ is a continuous process, once $S^*>0$, then  
	\begin{equation*}
	V\left(S^*\right)=\frac{\gamma\lambda}{1-\lambda\alpha^*}S^*+x^*\,.
	\end{equation*}
	Therefore, by multiplying both sides by $S^*$ and taking expectations  deduce that
	\begin{equation}\label{eq: sV(s)}
	ES^*V\left(S^*\right)=\frac{\gamma\lambda}{1-\lambda\alpha^*}E\left(S^*\right)^2+x^*ES^*=\frac{\gamma\lambda}{1-\lambda\alpha^*}E\left(S^*\right)^2+x^*\alpha^*\,.
	\end{equation}
	In addition, for every $u>0$ define a function $\upsilon(u)\equiv f\left(uS^*\right)$. It is known that $S^*$ is square-integrable and hence, for every $u>0$, $uS^*$ is also square-integrable. Thus, using \eqref{eq: f square-integrable} deduce that 
	\begin{equation}
	\upsilon(u)=E\int_0^{uS^*}V(s)ds-\frac{\gamma \lambda u^2E\left(S^*\right)^2}{2\left(1-\lambda u\alpha^*\right)}\,.
	\end{equation}
	Recall that $V(\cdot)$ is continuous on $\left[0,\infty\right)$ and hence the fundamental theorem of calculus implies that for every $u>0$ 
	\begin{equation}
	\frac{d}{du}\int_0^{uS^*}V(s)ds=S^*V(uS^*)\,.
	\end{equation}
	Observe that this derivative is nonnegative and dominated from above by $S^*V(0)$. Therefore, since $S^*$ is dominated by a linear function of $V(0)$ (see Lemma \ref{lemma: S bound}) and $EV^2(0)<\infty$, then Lemma \ref{lemma: leibnitz} allows replacing  the order of expectation and differentiation. Namely, $\upsilon(\cdot)$ is differentiable at some neighbourhood of $u=1$ with a derivative 
	\begin{equation}
	\frac{d\mu(u)}{du}\bigg|_{u=1}=ES^*V\left(S^*\right)-\gamma\left[\frac{\lambda E\left(S^*\right)^2}{\left(1-\lambda \alpha^*\right)}+\frac{\lambda^2\alpha^*E\left(S^*\right)^2}{2\left(1-\lambda \alpha^*\right)^2}\right]\,.
	\end{equation}
	It has already been shown that $u=1$ is a global maximum of $\upsilon(\cdot)$.
	Therefore, by applying this result with a first order condition at $u=1$ and an insertion of \eqref{eq: sV(s)} all together lead to a conclusion that
	\begin{equation}
	0=\frac{d\mu(u)}{du}\bigg|_{u=1}=\alpha^*\left[x^*-\gamma\frac{\lambda^2E\left(S^*\right)^2}{2\left(1-\lambda \alpha^*\right)^2}\right]\,.
	\end{equation}
	Note that $f\left(S^*\right)>0$ implies that $\alpha^*>0$ and hence the result follows.
\end{proof}
\newline\newline
The final step in the proof of Proposition \ref{lemma: existence} is to extend the result of Lemma \ref{lemma: step1} for a general $V(\cdot)$. 

\begin{lemma}\label{lemma: x identity}
	\begin{equation}
	x^*\equiv x_{\alpha^*}= \gamma\frac{\lambda^2E\left(S^*\right)^2}{2\left(1-\lambda E S^*\right)^2}\in(0,\infty)\,.
	\end{equation}
\end{lemma}
\begin{proof}
		Consider $V(\cdot)$ which might have jumps and for every $n\geq1 $ define $V_n(s)\equiv V\ast g_n(s),\forall s\geq0$ such that $g_n$ is given by the statement of Lemma \ref{lemma: approximation}. Note that due to this lemma, for each $n\geq1 $, $V_n$ satisfies the assumptions of Lemma \ref{lemma: step1}. In addition, by Lemma \ref{lemma: approximation}, it is known that for every $s\geq0$, $V_n(s)\uparrow V(s)$ as $n\to\infty$.  In addition, for every $n\geq1 $ consider the optimization
	\begin{equation} \label{optimization: social n}
	\begin{aligned}
	& \max_{S\in\mathcal{F}}:
	& &f_n(S)\equiv E\int_0^S\left[V_n(s)-s\frac{\gamma\lambda}{1-\lambda ES}\right]ds \\
	& \ \text{s.t:}
	& & 0\leq S \ , \ P \text{-a.s.}\,, \\ & & & ES<\lambda^{-1}\,.
	\end{aligned}
	\end{equation}
	and denote its objective functional by $f_n(\cdot)$. Note that for each $n\geq1$ \eqref{optimization: social n} has a solution $S_n$ such that
	\begin{equation}
	S_n=\inf\left\{s\geq0;V_n(s)-\frac{\lambda\gamma}{1-\lambda\alpha_n}s\leq x_n\right\}
	\end{equation}
	for some $\alpha_n\in\left(0,\lambda^{-1}\right)$ and $x_n\geq0$.
	Clearly, the sequence $ES_1,ES_2,\ldots$ is bounded on $\left[0,\lambda^{-1}\right)$. In addition, recall that for every $n\geq1 $, $S_n$ is bounded by a linear function of $V_n(0)\in\left[0,V(0)\right]$ (see also Lemma \ref{lemma: S bound}). Therefore, since $EV^2(0)<\infty$, the sequence $ES_1^2,ES_2^2,\ldots$ is bounded. Consequently,  there exists $\left\{n_k \right\}_{k=1}^\infty\subseteq\left\{1,2,\ldots\right\}$ such that
	\begin{equation}\label{eq: limits}
	\exists\lim_{k\to\infty}ES_{n_k }\equiv\alpha\ \ , \ \ \exists \lim_{k\to\infty}ES_{n_k }^2\equiv\sigma\,.
	\end{equation}
	
	For every $k\geq1$, $V_{n_k }(\cdot)$ is a process which satisfies the assumptions of Lemma \ref{lemma: step1} and hence 
	\begin{equation}
	x_{n_k }= \gamma\frac{ \lambda^2ES_{n_k }^2}{2\left[1- \lambda ES_{n_k }\right]^2}\geq0\ \ ,\ \ \forall k\geq1\,.
	\end{equation}
	This means that
	\begin{equation}
	\exists\lim_{k\to\infty}x_{n_k}=\gamma\frac{\lambda^2\sigma}{2\left(1-\lambda\alpha\right)^2}\equiv x
	\end{equation}
	such that $x^*\geq0$. Moreover, by construction, for every $s\geq0$, $V_{n_k }(s)\uparrow V(s)$ as $k\to\infty$ and observe that $\alpha_{n_k}\to\alpha$ as $k\to\infty$. Therefore, if for every $k\geq1$ 
	
	\begin{equation}
	\zeta_k(s)\equiv x_{n_k }+\frac{\gamma\lambda}{1-\lambda\alpha_{n_k }}s- V_{n_k }(s)\ \ , \ \ \forall s\geq0
	\end{equation}
	and 
	\begin{equation}
	\zeta(s)\equiv x+\frac{\gamma\lambda}{1-\lambda\alpha}s- V(s)\ \ , \ \ \forall s\geq0\,,
	\end{equation}
	then for every $s\geq0$, $\zeta_k(s)\rightarrow \zeta(s)$ as $k\to\infty$. Now, for every $k\geq1$ define 
	
	\begin{equation}
	\zeta^{-1}_k(u)\equiv\inf\left\{s\geq0;\zeta_k(s)\geq u\right\}\ \ , \ \ \forall u\in\mathbb{R}
	\end{equation}
	and 
	\begin{equation}
	\zeta^{-1}(u)\equiv\inf\left\{s\geq0;\zeta(s)\geq u\right\}\ \ , \ \ \forall u\in\mathbb{R}\,.
	\end{equation}
	Furthermore, observe that  $\zeta(\cdot),\zeta_1(\cdot),\zeta_2(\cdot),\ldots$ are all (strictly) increasing continuous processes tending to infinity as $s\to\infty$. Therefore, it can be deduced that $\zeta^{-1}(\cdot),\zeta^{-1}_1(\cdot),\zeta^{-1}_2(\cdot),\ldots$ are finite-valued continuous processes (with a time index $u$). Now, using exactly the same arguments which appear in the proof of Theorem 2A in \cite{Parzen1980}, deduce that for every $u\in\mathbb{R}$, $\zeta^{-1}_k(u)\rightarrow \zeta^{-1}(u)$ as $k\to\infty$. Especially, this is true for $u=0$, \textit{i.e.}, for every sample-space realization  
	\begin{equation}
	\exists\lim_{l\to\infty} S_{n_k }=\inf\left\{s\geq0;V(s)-\frac{\lambda\gamma}{1-\lambda\alpha}s\leq x\right\}\equiv S'\,.
	\end{equation}
	
	It is left to show that $S'$ is a solution of \eqref{optimization: social1}.  To this end, notice that for every $k\geq1$, $S_{n_k }$ is nonnegative and bounded from above by a linear function of $V_n(0)\leq V(0)$. Therefore, since $V(0)$ is square-integrable, dominated convergence implies that
	\begin{equation}
	\alpha=ES'\ \ , \ \ \sigma=E\left(S'\right)^2<\infty\,.
	\end{equation}
	To prove optimality of $S'$, observe that for every $k\geq1$,
	\begin{equation}
	V_{n_k }(s)\leq V(s)\ \ ,\ \ \forall s\geq0
	\end{equation}  
	and hence
	\begin{equation}
	f_{n_k }(S)\leq f(S)\ \ ,\ \ \forall S\in\mathcal{S}\,.
	\end{equation} 
	Thus, the optimality of $S_{n_k }$ (for each $k\geq1$) implies that  
	
	\begin{equation}
	f_{n_k }\left(S^*\right)\leq f_{n_k }\left(S_{n_k }\right)\leq f\left(S_{n_k }\right)\leq f\left(S^*\right)\ \ , \ \ \forall k\geq1\,.
	\end{equation}
	Moreover, it has already been shown that $S^*$ is square-integrable and hence it is possible to use \eqref{eq: f square-integrable}. Thus, since for every $s\geq0$, $0\leq V_{n_k }(s)\uparrow V(s)$ as $k\to\infty$, then monotone convergence implies that 
	\begin{align}
	f_{n_k }\left(S^*\right)&=E\int_0^{S^*}V_{n_k }(s)ds-\frac{\gamma\lambda}{2\left(1-\lambda\alpha^*\right)}E\left(S^*\right)^2\\&\xrightarrow{k\rightarrow\infty}E\int_0^{S^*}V(s)ds-\frac{\gamma\lambda}{2\left(1-\lambda\alpha^*\right)}E\left(S^*\right)^2=f\left(S^*\right)\,.\nonumber
	\end{align}
	In addition it is known that for every $k\geq1$, $ES_{n_k }^2<\infty$. Thus, it is possible to use \eqref{eq: f square-integrable} once again with a squeezing theorem in order to deduce that
	\begin{align}\label{eq: squeezing thm}
	f\left(S^*\right)&=\lim_{l\to\infty} f\left(S_{n_k }\right)\\&=\lim_{k\to\infty} E\int_0^{S_{n_k }}V(s)-\frac{\gamma\lambda \lim_{k\to\infty} E\left(S_{n_k }\right)^2}{2\left(1-\lambda\lim_{k\to\infty} \alpha_{n_k }\right)}\nonumber\\&=E\int_0^{S'}V(s)ds-\frac{\gamma\lambda E\left(S'\right)^2}{2\left(1-\lambda ES'\right)}\nonumber\\&=f\left(S'\right)\nonumber\,.
	\end{align} 
	In particular, notice that for every $k\geq1$, $S_{n_k }$ is bounded from above by $\frac{V(0)}{\gamma\lambda}$. Hence, for every $k\geq1$, $\int_0^{S_{n\left(k_l\right)}}V(s)ds$ is bounded from above by $\frac{V^2(0)}{\gamma\lambda}$ which is an integrable random variable. Thus, dominated convergence theorem justifies the third equality of  \eqref{eq: squeezing thm}. 
\end{proof}

\subsection*{Proof of Theorem \ref{thm: balking}}
Observe that $S^{**}$ and $V_{t_{\max}}(0)$ are square-integrable and hence
\begin{align}
0<v(t)&\leq P_t\int_0^{S^{**}}V_{t_{\max}}(0)ds=P_tES^{**}V_{t_{\max}}(0)\xrightarrow{t\uparrow t_{\text{max}}}0\,.
\end{align}
Since $t\mapsto P_t$ is positive and continuous on $t\in\left[t_{\min},t_{\max}\right)$, then it is enough to show that $r(t)\equiv v(t)/P_t$ is continuous on $\left[t_{\text{min}},t_{\text{max}}\right)$. To this end, fix $t\in\left[t_{\text{min}},t_{\text{max}}\right)$ and let $\{t_n\}_{n=1}^\infty\subseteq \left[t_{\text{min}},t+\epsilon\right]$ be a sequence such that $t_n\rightarrow t$ as $n\to\infty$ where $\epsilon\in(0,t_{\max}-t)$ is an arbitrary constant. Thus, since $(t,s)\mapsto\tilde{V}(t,s)$ is nondecreasing in its first coordinate, by Lemma \ref{lemma: S bound} deduce that for every $n\geq1$
\begin{equation}
0\leq S_{t_n}\leq\frac{V_{t_n}(0)}{\gamma\lambda P_{t_n}}\leq\frac{V_{t_{\max}}(0)}{\gamma\lambda P_{t+\epsilon}}\,.
\end{equation}
Since the upper bound is square-integrable and uniform in $n$, deduce that 
\begin{equation}
\alpha_n\equiv ES_{t_n}\ \  , \ \ \sigma_n\equiv ES_{t_n}^2\ \ , \ \ \forall n\geq1
\end{equation}
are two bounded sequences. Hence, there exists a subsequence $\left\{n(k)\right\}_{k=1}^\infty\subseteq\mathbb{N}$ such that

\begin{equation}
\exists\lim_{k\to\infty}\alpha_{n(k)}\equiv\alpha\ \ , \ \ \exists\lim_{k\to\infty}\sigma_{n(k)}\equiv\sigma\,.
\end{equation} 
In addition, Proposition \ref{lemma: existence} implies that for every $k\geq1$
\begin{equation}
S_{t_{n(k)}}=\inf\left\{s\geq0;V_{t_{n(k)}}(s)-s\frac{\gamma\lambda P_{t_{n(k)}}}{1-\lambda \alpha_{n(k)}P_{t_{n(k)}}}\leq x_{n(k)}\right\}
\end{equation}
such that 

\begin{equation}
0<x_{n(k)}=\frac{\gamma\left(\lambda P_{t_{n(k)}}\right)^2\sigma_{n(k)}}{2\left(1-\lambda\alpha_{n(k)} P_{t_{n(k)}}\right)}\xrightarrow{k\to\infty}\frac{\gamma\left(\lambda P_t\right)^2\sigma}{2\left(1-\lambda\alpha P_t\right)}\equiv x\,.
\end{equation}
In particular, note that $x$ is nonnegative. 

Now, let $U=F(T)\sim U(0,1)$ which is independent from $\tilde{V}$. In addition, observe that for every $t\in\left[t_{\min},t_{\max}\right)$  and $u\in\mathbb{R}$ 
\begin{equation}
P\left(T>u|T>t\right)=\frac{1-F(u\vee t)}{1-F(t)}\,.
\end{equation}
Therefore, for every $p\in(0,1)$ and $t\in\left[t_{\min},t_{\max}\right)$, the $p$'th quantile of $T$ given $\{T>t\}$ equals to 
\begin{equation}
q_t(p)=F^{-1}\left[1-(1-p)\left[1-F(t)\right]\right]\,.
\end{equation}
Recalling that $F(\cdot)$ is continuous and increasing on $\left[t_{\min},t_{\max}\right)$, then $q_t(p)$ is also continuous and increasing in $t$. In addition, without loss of generality, assume that $\tau_t=q_t(U)$. Consequently, since $(t,s)\mapsto\tilde{V}(t,s)$ is continuous and nondecreasing in $t$, then $(t,s)\mapsto V_t(s)$ is also continuous and nondecreasing in $t$. Thus, by applying the same technique which appears in the proof of Lemma \ref{lemma: x identity}, deduce that 
\begin{equation}\label{eq: S' definition}
\exists\lim_{k\to\infty}S_{t_{n(k)}}=\inf\left\{s\geq0;V_{t}(s)-s\frac{\gamma\lambda P_t}{1-\lambda P_t}\leq x\right\}\equiv S'_t\,.
\end{equation} 
In addition, since $S_t\leq S^{**}$ for every $t\in\left[t_{\min},t_{\max}\right)$ and $S^{**}$ is square-integrable, then dominated convergence implies that
\begin{equation}
\alpha= ES'_t\ \ , \ \ \sigma=E\left(S'_t\right)^2\,.
\end{equation}
For every $\nu_1,\nu_2\in\left[t_{\min},t_{\max}\right)$ define 
\begin{equation}
w_{\nu_1}(\nu_2)\equiv E\int_0^{S_{\nu_2}}V_{\nu_1}(s)ds-\frac{\gamma\lambda P_{\nu_1} ES_{\nu_2}^2}{2\left(1-\lambda P_{\nu_1} ES_{\nu_2}\right)}
\end{equation}
which is nondecreasing in $\nu_1$ and such that $r(\nu_1)= w_{\nu_1}(\nu_1)$. In addition, notice that
\begin{equation}
w_{t_{n(k)}}(t)\leq w\left[t_{n(k)}\right]\leq w_{t+\epsilon}\left[t_{n(k)}\right]\ \ , \ \forall k\geq1\,.
\end{equation}
Moreover, since for every $k\geq1$
\begin{equation}
0\leq V_{t_{n(k)}}(s)1_{\left[0,S_t\right]}(s)\leq V_{t_{\max}}(0)1_{\left[0,S^{**}\right]}(s)\,,
\end{equation}
then dominated convergence implies that
\begin{align}
w_{t_{n(k)}}(t)&=E\int_0^{S_{t}}V_{t_{n(k)}}(s)ds-\frac{\gamma\lambda P_{t_{n(k)}} ES_{t}^2}{2\left[1-\lambda P_{t_{n(k)}} ES_{t}\right]}\\&\xrightarrow{k\to\infty}E\int_0^{S_{t}}V_{t}(s)ds-\frac{\gamma\lambda P_t ES_{t}^2}{2\left(1-\lambda P_t ES_{t}\right)}=r(t)\,.
\end{align}
Similarly, since for every $k\geq1$ 
\begin{equation}
0\leq \int_0^{S_{t_{n(k)}}}V_{t+\epsilon}(s)ds\leq S^{**}V_{t_{\max}}(0)\,,
\end{equation}
dominated convergence might be used once again in order to derive the limit

\begin{align}
w_{t+\epsilon}\left[t_{n(k)}\right]&=E\int_0^{S_{t_{n(k)}}}V_{t+\epsilon}(s)ds-\frac{\gamma\lambda P_{t+\epsilon} ES_{t_{n(k)}}^2}{2\left(1-\lambda P_{t+\epsilon} ES_{t_{n(k)}}\right)}\\&\xrightarrow{k\to\infty}E\int_0^{S'_{t}}V_{t+\epsilon}(s)ds-\frac{\gamma\lambda P_{t+\epsilon} E\left(S_t'\right)^2}{2\left(1-\lambda P_{t+\epsilon} ES'_{t}\right)}=\frac{\phi_{t+\epsilon}(S'_t)}{P_{t+\epsilon}}\nonumber
\end{align}
Therefore, deduce that
\begin{equation}
r(t)\leq \lim_{k\to\infty}\inf w\left[t_{n(k)}\right]\leq \lim_{k\to\infty}\sup w\left[t_{n(k)}\right]\leq \frac{\phi_{t+\epsilon}(S'_t)}{P_{t+\epsilon}}\,.
\end{equation}
Note that this inequality is valid even when $\epsilon$ is replaced by some $\epsilon'\in(0,\epsilon)$. This is true because up to a finite prefix, the sequence $\{t_{n(k)}; k\geq1\}$ belongs to $(t_{\min},t+\epsilon')$. Thus, the next step is to take a limit of the upper bound as $\epsilon\downarrow0$. In practice, the same kind of arguments which were made in the previous limits calculations, imply that 
\begin{align}
\frac{\phi_{t+\epsilon}(S'_t)}{P_{t+\epsilon}}&=E\int_0^{S'_{t}}V_{t+\epsilon}(s)ds-\frac{\gamma\lambda P(T>t+\epsilon) E\left(S_t'\right)^2}{2\left[1-\lambda P\left(T>t+\epsilon\right) ES'_{t}\right]}\nonumber\\&\xrightarrow{\epsilon\downarrow0}E\int_0^{S'_{t}}V_{t}(s)ds-\frac{\gamma\lambda P(T>t) E\left(S_t'\right)^2}{2\left[1-\lambda P\left(T>t\right) ES'_{t}\right]}=r(t)\,.
\end{align}
This shows that $r(t_{n(k)})\rightarrow r(t)$ as $k\to\infty$. 

Now, note that $t\mapsto P_t$ is decreasing on $[t_{\min},t_{\max}]$ and for every $0<p_1<p_2\leq1$, $\mathcal{D}_{p_2}\subset\mathcal{D}_{p_1}$. Therefore, since for every $s\geq0$, $t\mapsto V_t(s)$ is nondecreasing on $[t_{\min},t_{\max}]$, then 

\begin{equation}
r(t)=\max_{S\in\mathcal{D}_{P_t}}\left\{ E\int_0^{S}V_{t}(s)ds-\frac{\gamma\lambda P(T>t) ES^2}{2\left[1-\lambda P\left(T>t\right) ES\right]}\right\}
\end{equation}
is increasing in $t$ on $(t_{\min},t_{\max})$. Furthermore, for every $t\in(t_{\min},t_{\max})$
\begin{equation}
0<r(t)\leq E\int_0^{S^{**}}V_{t_{\max}}(0)ds=ES^{**}V_{t_{\max}}(0)<\infty\,.
\end{equation}
Hence, if either $t_n\uparrow t$ as $n\to\infty$ or $t_n\downarrow t$ as $n\to\infty$, then $\{r(t_n);n\geq1\}$ is a bounded monotone sequence. This means that in both cases
\begin{equation}
\exists\lim_{n\to\infty}r(t_n)=\lim_{k\to\infty}r(t_{n(k)})=r(t)\,.
\end{equation}
Hence, by the generality of $\{t_n\}_{n=1}^\infty$, deduce that
\begin{equation}
\lim_{u\uparrow t}r(u)=\lim_{u\downarrow t}r(u)=r(t)
\end{equation}
which makes the result follows.
\end{document}